\documentclass[11pt,a4paper,english,reqno]{amsart}

\usepackage{verbatim}
\usepackage{amsthm}
\usepackage{amscd}
\usepackage{amstext}
\usepackage{amssymb}
\usepackage[all]{xy}
\usepackage[dvips]{graphicx}
\usepackage[colorinlistoftodos]{todonotes}
\usepackage[colorlinks=true,linkcolor=blue,citecolor=blue]{hyperref}
\usepackage{tabu}
\usepackage{mathdots}

\usepackage{tikz-cd}
\usepackage{tikz}
\usetikzlibrary{positioning,intersections,matrix,arrows,backgrounds}
\usetikzlibrary{shapes,shapes.geometric,shapes.misc}
\usetikzlibrary{decorations,decorations.pathreplacing}

\usepackage{pifont}
\usepackage{color}
\usepackage{type1cm}
\usepackage{stmaryrd}
\usepackage{siunitx}
\usepackage{calrsfs}
\usepackage{eucal}
\usepackage{mathptmx}
\usepackage{mathrsfs}

\makeatletter
\@namedef{subjclassname@2010}{%
    \textup{2010} Mathematics Subject Classification}
\makeatother

\newcommand{\bsm}{\begin{smallmatrix}}
\newcommand{\esm}{\end{smallmatrix}}
\newcommand{\bbm}{\begin{matrix}}
\newcommand{\ebm}{\end{matrix}}
\newcommand{\rt}{\rightarrow}
\newcommand{\lrt}{\longrightarrow}

\newcommand{\st}{\stackrel}

\newcommand{\la}{\lambda}
\newcommand{\La}{\Lambda}

\newcommand{\mmod}{{\mathrm{mod}}\mbox{-}}
\newcommand{\CA}{\mathscr{A} }
\newcommand{\CC}{\mathscr{C} }
\newcommand{\DD}{\mathscr{D} }
\newcommand{\CF}{\mathscr{F} }

\newcommand{\CP}{\mathscr{P }}

\newcommand{\CS}{\mathscr{S} }
\newcommand{\CT}{\mathscr{T} }

\newcommand{\CM}{\mathscr{M} }

\newcommand{\CH}{\mathscr{H} }

\newcommand{\CE}{\mathscr{E} }

\newcommand{\CV}{\mathcal{V} }

\newcommand{\Hom}{\rm{Hom}}
\newcommand{\Ext}{\rm{Ext}}

\newtheorem{theorem}{Theorem}[section]
\newtheorem{corollary}[theorem]{Corollary}
\newtheorem{lemma}[theorem]{Lemma}
\newtheorem{proposition}[theorem]{Proposition}

\newtheorem{notation}[theorem]{Notation}
\theoremstyle{definition}
\newtheorem{definition}[theorem]{Definition}
\newtheorem{example}[theorem]{Example}

\newtheorem{construction}[theorem]{Construction}
\newtheorem{remark}[theorem]{Remark}

\theoremstyle{plain}

\theoremstyle{definition}

\numberwithin{equation}{section}

\newtheorem*{theorem a*}{Theorem A}
\newtheorem*{theorem b*}{Theorem B}
\newtheorem*{theorem c*}{Theorem C}
\newtheorem*{theorem d*}{Theorem D}
\newtheorem*{theorem e*}{Theorem E}
\begin{document}

\title[$r$ICE-closed subcategories]{$r$ICE-closed subcategories induced by the   morphism category of projective modules }

\author{ Rasool Hafezi }
\dedicatory{}
\address{School of Mathematics and Statistics,
Nanjing University of Information Science \& Technology, Nanjing, Jiangsu 210044, P.\,R. China}
\email{hafezi@nuist.edu.cn}
\author{Alireza Nasr-Isfahani}
\address{Department of Pure Mathematics\\
Faculty of Mathematics and Statistics\\
University of Isfahan\\
P.O. Box: 81746-73441, Isfahan, Iran\\ and School of Mathematics, Institute for Research in Fundamental Sciences (IPM), P.O. Box: 19395-5746, Tehran, Iran}
\email{nasr$_{-}$a@sci.ui.ac.ir / nasr@ipm.ir}
\author{Jiaqun Wei}
\address{School of Mathematical Science,
Zhejiang Normal University, Jinhua 321004, China}
\email{weijiaqun5479@zjnu.edu.cn}


\begin{abstract}
Let $\Lambda$ be an Artin $R$-algebra, and  ${\rm proj}\mbox{-}\Lambda$ denotes the category of all finitely generated projective $\Lambda$-modules. Define $\CP(\Lambda) := {\rm Mor}({\rm proj}\mbox{-}\Lambda)$. Due to the favorable homological properties of $\CP(\Lambda)$, we initially examine several noteworthy objects and subcategories of $\CP(\Lambda)$, subsequently relating these findings to $\mmod \Lambda$. Following our examination of Image-Cokernel-Extension closed (hereafter referred to as ICE-closed) subcategories of $\CP(\Lambda)$, among other bijections, we demonstrate a bijection between rigid objects in $\CP(\Lambda)$ and ICE-closed subcategories of $\CP(\Lambda)$ with enough Ext-projectives. In order to translate the concept of ICE-closed subcategory from $\CP(\Lambda)$ to $\mmod \Lambda$, it is necessary to introduce the framework of rICE-closed subcategories of $\mmod \Lambda$. We then establish a bijection between $\tau$-rigid modules in $\mmod \Lambda$ and rICE-closed subcategories of $\mmod \Lambda$ that possess an rExt-progenerator. This is a generalization of a bijection given by Enomoto for hereditary algebras. Our morphism approach improves a bijection given by Buan and Zhou by introducing r-cotorsion-torsion triples. We conclude our paper with further applications for $\tau$-tilting theory.
\end{abstract}

\maketitle


\section{Introduction}\label{Introduction}

Let $R$ be a commutative Artinian ring and $\Lambda$ be an Artin $R$-algebra. 
The morphism category $\mathrm{Mor}(\Lambda)$ consists of two-term complexes 
of $\Lambda$-modules concentrated in homological degrees $1$ and $0$, with 
morphisms given by morphisms of complexes. It is known that the category 
$\mathrm{Mor}(\Lambda)$ is equivalent to the category of left modules over 
the lower triangular matrix algebra $T_2(\Lambda)$ (for more details see 
\cite[Chapter I.4]{Aus} and \cite[Chapter III.2]{ARS}).  Let ${\rm proj}\mbox{-}\La$ be the category of all finitely generated projective $\La$-modules and set $\CP(\La):={\rm Mor}({\rm proj}\mbox{-}\La).$ Then $\CP(\La)$ is a full and extension-closed subcategory of 
$\mathrm{Mor}(\Lambda)$ which has good homological properties 
(see \cite{Ba}). The category $\CP(\La)$ is naturally equipped with the 
structure of an exact category induced by the exact structure of 
$\mathrm{Mor}(\Lambda)$, where conflations are defined as degree-wise 
exact sequences. Moreover, $\CP(\La)$ is hereditary with enough 
projective and injective objects. 

Recent work of the first and third authors shows that the representation 
theory of $\mmod \La$ and $\CP(\La)$ are closely related. By \cite{HW}, 
we know that $\CP(\La)$ has almost split sequences. The first and third 
authors in \cite{HW} showed that the functor $\mathsf{Cok}:\CP(\La) \rt \mmod \La$
is full, dense, and objective, which induces an equivalence $\overline{\CP(\La)} \cong \mmod \La,$ where $\overline{\CP(\La)}$ is the injectively stable category of 
$\CP(\La)$. 

By taking $\CC={\rm proj}\mbox{-}\La$ in \cite[Theorem 4.2]{HE}, we infer 
that the Auslander--Reiten quiver of $\La$ is a full subquiver of the 
Auslander--Reiten quiver of $\CP(\La)$. This shows that the difference 
between their Auslander--Reiten quivers consists only of indecomposable 
injective objects in $\CP(\La)$, which are completely characterized. 
Therefore, the good homological properties of $\CP(\La)$ suggest that one 
may translate information from $\mmod \La$ to $\CP(\La)$, analyze this 
information using the homological properties of $\CP(\La)$, and then 
transfer the results back to $\mmod \La$. A similar observation was made in Auslander's work, where he considered  the category $((\mmod \La)^{\rm op}, \mathcal{A}{\rm b})$
of  finitely presented contravariantly functors from $\mmod \La$ to the 
category of abelian groups $\mathcal{A}{\rm b}$, which has global dimension $2$, 
to study certain properties of $\mmod \La$. In contrast, $\CP(\La)$ is an 
exact category of global dimension one, so from a homological perspective 
it is ``one dimension smaller''. 

In this paper, we plan to use $\CP(\La)$ as a tool for studying 
$\tau$-tilting theory in $\mmod \La$, instead of using the homotopy 
category of two-term complexes with projective components, as is commonly 
done in the literature.

Tilting theory is one of the most powerful tools in the representation theory of finite-dimensional algebras. Brenner and Butler introduced tilting modules \cite{BB} (see also \cite{HR}). It is well known that an almost complete tilting module has at most two complements \cite{RS, Un}, and Happel and Unger \cite{HU} showed that there are exactly two complements in the faithful case. The idea of two complements is closely related to mutation in cluster algebras \cite{FZ}. Later, Derksen and Fei \cite{DF} observed a class of representations with similar behavior. Adachi, Iyama, and Reiten introduced $\tau$-tilting modules and developed their theory in \cite{AIR}. Nowadays, $\tau$-tilting theory has become an active area in the representation theory of finite-dimensional algebras.

On the other hand, certain subcategories of finitely generated modules also play an important role in representation theory. In particular, torsion classes and torsion-free classes are closely related to tilting theory. Moreover, there are several bijections between classes of subcategories and well-structured modules (see \cite{IT} and references therein). For example, Adachi, Iyama, and Reiten \cite{AIR} established a bijection between functorially finite torsion classes and isomorphism classes of basic support $\tau$-tilting modules.

ICE-closed subcategories, that is, subcategories closed under images, cokernels, and extensions, were introduced by Enomoto in \cite{En} as a common generalization of torsion classes and wide subcategories. In \cite{En}, for a Dynkin quiver $Q$, a bijection between ICE-closed subcategories of $\mathrm{mod}\,kQ$ and rigid $kQ$-modules was established, extending earlier bijections by Adachi, Iyama, and Reiten, and by Ingalls and Thomas \cite{IT}. 

Later, Enomoto and Sakai generalized this result in two directions. First, in \cite{ES2}, they introduced twin rigid modules to study IE-closed subcategories, namely subcategories closed under images and extensions. In the hereditary case, they obtained a bijection between functorially finite IE-closed subcategories and isomorphism classes of twin rigid modules.

Second, in \cite{ES}, they considered ICE-closed subcategories that are functorially finite torsion classes in functorially finite wide subcategories, called \emph{doubly functorially finite}. Over an Artin algebra, they established a bijection between such subcategories and isomorphism classes of basic wide $\tau$-tilting modules, where a wide $\tau$-tilting module means a $\tau_{\mathcal W}$-tilting module in a functorially finite wide subcategory $\mathcal W$ (see \cite[Definition 4.11]{ES}).

In this paper, we study certain objects and subcategories in $\CP(\La)$ and establish a series of bijections between them. We then transfer these results to $\mmod \La$. In this way, we obtain new results for $\mmod \La$ and show that several known results can be derived from $\CP(\La)$. In particular, this leads to the notion of rICE-closed subcategories in module categories.

Sauter \cite{S} introduced tilting objects and tilting subcategories in exact categories with enough projectives (see Definition \ref{Def. tilting. CP}). Pan, Zhang, and Zhu \cite{PZZ} defined and studied (support) $\tau$-tilting objects and subcategories in exact categories. More generally, the authors in \cite{GNP1} studied hereditary extriangulated categories and developed mutation theory for silting objects in certain hereditary extriangulated categories, called $0$-Auslander extriangulated categories. Typical examples include homotopy categories of two-term complexes of projectives. As discussed in \cite{HW}, another source of $0$-Auslander extriangulated categories is the exact category $\CP(\La)$. Since silting objects in $\CP(\La)$ coincide with tilting objects, the theory developed in \cite{GNP1} applies to the hereditary exact category $\CP(\La)$.

One of our strategies is to find analogues of results known for module categories over hereditary algebras in the hereditary exact category $\CP(\La)$. Our first main result is an analogue of \cite[Theorem 2.3]{En}, which gives, for hereditary algebras, a bijection between rigid modules and ICE-closed subcategories with enough $\Ext$-projectives.

\medskip

\textbf{Theorem A.} (Theorem \ref{Thm. ICE-rigid-Mor}) {\it Let $\Lambda$ be an  Artin algebra. Then there exist mutually inverse bijection between the following two sets:
\begin{itemize}
    \item [$(1)$] isomorphism class of basic rigid objects in $\CP(\La);$
    \item [$(2)$] ICE-closed subcategories with enough $\Ext$-projectives, see Definition \ref{Def. ICE-Mor-Proj}. 
\end{itemize}
 }

ICE-closed subcategories in $\CP(\Lambda)$ are defined similarly to those in $\mmod \Lambda$. 
However, in this setting we should work with admissible morphisms, namely morphisms 
which factor as an inflation followed by a deflation. Consequently, the cokernel, 
kernel, and image of an admissible morphism in the exact category $\CP(\Lambda)$ 
are well-defined. Rigid objects in $\CP(\Lambda)$ are defined in the usual way, 
namely as objects having no self-extensions. 

The map from $(1)$ to $(2)$ in the above theorem is given by sending a rigid object 
$\mathsf{M}$ to $\mathsf{Cok}\,\mathsf{M}$, the subcategory of $\CP(\Lambda)$ 
consisting of cokernels of admissible morphisms in $\mathsf{add} \mathsf{M}$.

One advantage of working with $\CP(\Lambda)$ is the existence of a duality between 
$\CP(\Lambda)$ and $\CP(\Lambda^{\mathrm{op}})$, obtained by applying 
$(-)^*=\Hom_{\Lambda}(-,\Lambda)$ componentwise to objects in $\CP(\Lambda)$. 
This yields a dual version of Theorem~A stated in 
\ref{Thm. IKE-rigid-Mor}. In the dual setting, we deal with IKE-closed 
subcategories in $\CP(\Lambda)$, that is, subcategories closed under  
admissible images, admissible kernels and extensions. 

According to Theorem~A and its dual, we observe that for a pair of rigid objects 
$(\mathsf{P},\mathsf{I})$, the subcategory 
$\mathsf{Cok}\,\mathsf{P}\cap \mathsf{Ker}\,\mathsf{I}$ 
is an IE-closed subcategory, closed under admissible images and extensions. 
Inspired by twin modules defined in \cite{ES2}, we introduce in 
Definition~\ref{def:twin-rigid} the notion of twin objects, which are pairs 
of rigid objects satisfying certain homological conditions. We then use them 
to classify IE-subcategories as follows.

\medskip

\textbf{Theorem B.} (Theorem \ref{thm:main-twin-IEclosed123})
 Let $\La$ be an Artin algebra. Then there is mutually inverse   bijections between the following two sets:
  \begin{enumerate}
    \item The set of  equivalent classes of IE-closed subcategories $\CC$ of $\CP(\La)$ which  admit an $\Ext$-progenerator  and an $\Ext$-injective cogenerator.
    \item The set of isomorphism classes of basic twin rigid objects $(\mathsf{P}, \mathsf{I})$.
  \end{enumerate}

We should point out a similar version of Theorem B is proved in \cite[Theorem 2.14]{ES2} over hereditary algebras.

Next, we aim to develop a version of the above results in $\CP(\Lambda)$ for the 
module category over an  Artin algebra. In this paper, we mainly focus on 
Theorem~A. Although Theorem~B may lead to interesting notions in the module category, 
its investigation requires an independent study. According to Theorem~A, we seek to 
identify which modules arise as images under the cokernel functor applied to rigid 
objects. 

We will show in Lemma~\ref{rigid-tau-rigid} that rigid objects in $\CP(\Lambda)$ are closely related 
to $\tau$-rigid modules in $\mmod \Lambda$. More generally, we prove that there is 
a bijection between tilting objects in $\CP(\Lambda)$ and support $\tau$-tilting modules 
in $\mmod \Lambda$. This bijection explains how $\tau$-tilting theory originates 
from tilting theory in the morphism category of projective modules. One may further 
pursue these ideas inside $\CP(\Lambda)$. For instance, the bijection given in 
\cite[Theorem~2.14]{AIR} between support $\tau$-tilting modules in $\mmod \Lambda$ 
and those in $\mmod \Lambda^{\rm op}$ is induced by the duality  $
(-)^*:\CP(\Lambda)\longrightarrow \CP(\Lambda^{\rm op}).$
Moreover, the bijection between support $\tau$-tilting modules and support 
$\tau^{-1}$-tilting modules can also be reproved using our morphism approach; 
please see the preliminary version of this paper \cite{HNW} for further details. Also, mutation 
of support $\tau$-tilting modules can be formulated via mutation of tilting (or 
silting) objects in $\CP(\Lambda)$, based on \cite{GNP1}. We also note that in 
\cite{MS}, the morphism category of all modules, rather than only projective modules 
as in the present paper, is used to translate silting modules into tilting objects 
for the study of universal localization.

To achieve our goal, we need to analyze ICE-closed subcategories of $\CP(\Lambda)$ and then 
transfer the results to $\mmod \Lambda$. This leads to the notion of 
{\em rICE-closed subcategories}. In the crucial 
Construction~\ref{Construction5}, we explain how a right exact sequence in 
$\mmod \Lambda$ is induced by a short exact sequence in $\CP(\Lambda)$.  Let $(A_3,J)$ be the bound quiver  $
A_3: v_2 \xrightarrow{a} v_1 \xrightarrow{b} v_0$
with relation $J$ generated by $ab$, and let ${\rm rep}(A_3,J,\Lambda)$ denote 
the category of representations of $(A_3,J)$ by $\Lambda$-modules and 
$\Lambda$-homomorphisms. Viewing a right exact sequence as an object in 
${\rm rep}(A_3,J,\Lambda)$, the above construction produces a (minimal) projective 
presentation of the right exact sequence in the abelian category 
${\rm rep}(A_3,J,\Lambda)$. Taking the kernel of this minimal projective presentation 
yields information about the second syzygies of right exact sequences in 
${\rm rep}(A_3,J,\Lambda)$.  This leads to Definition~\ref{def. typ0-1-P-right-exact}, where we introduce two  types of right exact sequences, called $P$-right exact sequences of type $0$ and  type $1$. If both conditions are satisfied, we call the sequence a $P$-right exact sequence. In Definition~\ref{Def. ICE-Mor-Proj}, we formally introduce rICE-closed subcategories, which are roughly defined as subcategories closed under  images and cokernels of 
$P$-right exact sequences, as well as extensions of $P$-right exact sequences. Over a hereditary algebra, rICE-closed subcategories 
coincide with ICE-closed subcategories.

We are now ready to state a version of Enomoto's theorem 
\cite[Theorem~2.3]{En} for general Artin algebras.

\textbf{Theorem C.} (Theorem \ref{mainTheorem5})
Let $\La$ be an Artin algebra.
Then there exist mutually inverse bijections between the following two sets:
\begin{itemize}
    \item[(1)] Isomorphism classes of $\tau$-rigid modules;
    \item[(2)]  rICE-closed subcategories of $\mmod \La$ that admit an
    $r\Ext$-progenerator.
\end{itemize}

According to \cite[Theorem~3.8]{PZZ} or \cite[Theorem~5.2]{GNP1}, we obtain in 
Proposition~\ref{prop. bij-twin-tilting-cotortor} a bijection between tilting objects 
and cotorsion--torsion triples in $\CP(\Lambda)$. Recall that a cotorsion--torsion 
triple consists of a cotorsion pair together with a torsion pair. 

In view of the setting of our paper, a cotorsion--torsion triple in $\CP(\Lambda)$ 
corresponds to an r-cotorsion--torsion triple in $\mmod \Lambda$, consisting of an 
r-cotorsion pair and a torsion pair; see 
Definition~\ref{Def. r-cot-r-tor-triple2}. We thus obtain the following bijection.

\medskip

\noindent
\textbf{Theorem D.} (Theorem~\ref{Thm. r-cot-r-torsion-supp})
There exists a bijection between the following sets:
\begin{itemize}
\item[$(1)$] Isomorphism classes of basic support $\tau$-tilting modules;
\item[$(2)$] r-cotorsion--torsion triples.
\end{itemize}

\medskip

In \cite{BZ}, Buan and Zhou introduced the notion of a left weak cotorsion pair 
(for short, lw-cotorsion pair), consisting of a left weak cotorsion pair together 
with a torsion pair; see Definition~\ref{def. lw-cotrsion-pair}. They proved a 
bijection between lw-cotorsion triples and support $\tau$-tilting modules. 
In contrast to an lw-cotorsion pair, the subcategories involved in an r-cotorsion 
pair determine each other.

The paper is organized as follows. In Section~2, we provide the necessary background 
needed throughout the paper. In Section~3, we introduce ICE-closed, IKE-closed, and 
IE-closed subcategories in $\CP(\Lambda)$ and prove the bijections stated in 
Theorems~A and~B.  In Section~4, we study cotorsion--torsion triples and their duals (torsion--cotorsion  triples) in $\CP(\Lambda)$, and establish a correspondence between these triples and 
tilting objects in $\CP(\Lambda)$. Section~5 builds a bridge between our results in 
the morphism category $\CP(\Lambda)$ and the module category $\mmod \Lambda$. In this 
section, we introduce $P$-right exact sequences, which form the main ingredient of 
our notion of rICE-closed subcategories defined in the next section.  In Section~6, we prove Theorem~C. In the final section, we prove Theorem~D and establish a bijection between tilting objects in $\CP(\Lambda)$ and support $\tau$-rigid modules in $\mmod \Lambda$. We conclude the paper with further applications, in particular Theorem~\ref{final}, which illustrates the role of right exact sequences in support $\tau$-tilting theory.

Throughout this paper, all subcategories are assumed to be full and closed under isomorphisms. An Artin algebra  $\La$ is assumed to be basic.

\section{Preliminaries}\label{Preleminary}
In this section, we collect some necessary background material for the paper.

\subsection{Morphism category}
Let $\CC$ be a category.
The morphism category $\rm{Mor}(\CC)$ of $\CC$ is a category whose objects are morphisms $f:X\rt Y$ in $\CC$, and whose morphisms are given by commutative diagrams. If we regard the morphism $f:X\rt Y$ as an object in $\rm{Mor}(\CC)$, we will usually present it as
 $(X\st{f}\rt Y)$.
A morphism between the objects $(X\st{f}\rt Y)$ and $(X'\st{f'}\rt Y')$ is presented as
$(\sigma_1, \sigma_2): (X\st{f}\rt Y)\rt (X'\st{f'}\rt Y')$,  where $\sigma_1:X\rt X'$ and $\sigma_2:Y\rt Y'$ are morphisms in $\CC$ with $\sigma_2f=f'\sigma_1.$ For breviary,  we let $\CH(\La):={\rm Mor}(\mmod \La)$ and $\CP(\La):={\rm Mor}({\rm proj}\mbox{-}\La)$. If there is no possibility of confusion, we write $\CH$ and $\CP$, respectively. We typically use a sans-serif font style, for example, $\mathsf{A}, \mathsf{B}, \ldots,$ to denote objects in the morphism category.

\subsection{ Exact categories}

 Let $\mathscr{E}=(\mathscr{A},\mathscr{S})$ be an exact category, where $\mathscr{A}$ is an additive category and $\mathscr{S}$ is a class of kernel-cokernel pairs which satisfies axioms of \cite[Definition 2.1]{Bu}. We call an element $X\stackrel{f}\rightarrow Y\stackrel{g}\rightarrow Z$ in $\mathscr{S}$ a conflation, $f$ an inflation and $g$ a deflation. We depict inflations (resp. deflations) by $\rightarrowtail$ (resp. $\twoheadrightarrow$). A morphism $f$ in $\mathscr{E}$ is called {\em admissible} if $f=i\circ d$ for some inflation $i$ and deflation $d$. A sequence of admissible morphisms $X\stackrel{f} \rightarrow Y\stackrel{g} \rightarrow Z$ in $\mathscr{E}$ is {\em exact} at $Y$ if Im$f$=Ker$g$. Here we emphasize that, since $f$ is an admissible morphism, the notions of $\mathsf{Ker}$, $\mathsf{Cok}$, and $\mathsf{Im}$ are well defined. Of course, these notions cannot be defined for an arbitrary morphism in an exact category.

 For an additive subcategory $\mathscr{T}$ of $\CE$, define
\begin{itemize}
\item $\mathsf{Fac}\mathcal{T}=\{X\in \mathcal{E}~|~\exists~T\twoheadrightarrow X,~T\in \mathcal{T}\}$ 
\item $ \mathsf{Sub}\mathcal{T}=\{X\in \mathcal{E}~|~\exists~X\rightarrowtail T,~T\in \mathcal{T}\}$
\item $\mathsf{Cok}\mathcal{T}=\{X\in \mathcal{E}~|~\exists~\text{admissible morphism}~f \in \mathcal{T}~\text{such that}~\mathsf{Cok}f=X \}$
\item $\mathsf{Ker}\mathcal{T}=\{X\in \mathcal{E}~|~\exists~\text{admissible morphism}~f \in \mathcal{T}~\text{such that}~\mathsf{Ker}f=X \}$

\end{itemize}

For an object $X \in \CE$, set  $\mathsf{Fac} X:=\mathsf{Fac}(\mathsf{add} X)$ and $\mathsf{Sub} X:=\mathsf{Sub}(\mathsf{add} X)$, where $\mathsf{add} X$ is the smallest additive subcategory of $\CE$ which contains $X$.

By using projective (or injective) resolutions, as in the abelian case, there are abelian groups ${\rm Ext}^{n}_{\mathscr{E}}(X, Y)$ for $n\geq 0$ and $X, Y\in \mathscr{E}$. We refer to \cite{K22} for details. 
A \emph{full exact subcategory} $\mathscr{C}$ of $\mathscr{E}$ is a full additive subcategory that is closed under extensions.  In this case, $(\mathscr{C},\mathscr{S}|_{\mathscr{C}})$ is an exact category. Note that the subcategory $\mathsf{P}(\mathscr{C})$ consisting of Ext-projective objects in $\mathscr{C}$ (i.e. $X\in \mathscr{C}$ such that ${\rm Ext}_{\mathscr{E}}^{1}(X,\mathscr{C})=0$) is precisely the subcategory of projective objects in the exact category $\mathscr{C}$. In a dual manner, we can define the subcategory $\mathsf{I}(\mathscr{C})$ consisting of Ext-injective objects in $\mathscr{C}$.

\begin{definition}
 Let $\CE$ be an  exact category, and let  $\CC$ be  an extension-closed subcategory of $\CE$.
   An object $P \in \CC$ is an \emph{$\Ext$-progenerator} if $\mathsf{P}(\CC)=\mathsf{add} P$ and $\CC$ has enough $\Ext$-projectives, that is,  for every object $X \in \CC$, there is a conflation
$$Y\rightarrowtail P\twoheadrightarrow X$$
      with $P \in \mathsf{P}(\CC)$ and $Y \in \CC$. Dually, one can define  \emph{enough $\Ext$-inejective} and \emph{$\Ext$-injective cogenerator}  for $\CC.$
\end{definition}
 If $\CE$ has the Krull--Schmidt property, we may assume that the $\Ext$-progenerator or $\Ext$-injective cogenerator  in the above definition are basic, and hence uniquely determined up to isomorphism.  
In this case, we denote these basic objects by $ \mathsf{P}(\CC)$ and $\mathsf{I}(\CC)$, respectively.
\subsection { Auslander-Reiten theory in $\CP(\La)$}\label{Sub. remainder}
Since in this paper we are mostly dealing with the exact category $\CP(\La)$. Here, we specialize to this exact category.  From \cite[Proposition 3.6]{HE}, $\CP(\La)$ has almost split sequences. Therefore, by  \cite{INP}, $\CP(\La)$ has ARS-duality. which 
 is a pair $(\tau_{\CP}, \eta)$ consisting of an equivalence $\tau_{\CP}:\underline{\CP(\La)}\rt \overline{\CP(\La)}~$ together with a bi-natural isomorphism
$$\eta_{\mathsf{X, Y}}:{\Hom}_{\underline{\CP(\La)}}(\mathsf{X}, \mathsf{Y})\simeq D{\Ext}^1_{\CP}(\mathsf{Y}, \tau_{\CP}(\mathsf{X})) \ \ \ \text{for any}  \ \ \mathsf{X}, \mathsf{Y} \in \CP(\La).$$

 Throughout the paper, the cokernel functor  $\mathsf{Cok}:\CP(\La)\rt \mmod \La$ is defined by sending  $(P\st{f}\rt Q)\mapsto \mathsf{Cok}f.$

\begin{theorem} \cite[Theorem 4.2]{HE}\label{Thm-Cok-equiv}
The functor $\mathsf{Cok}:\CP(\La) \rt \mmod \La$ is full, dense and objective.  Thus, the cokernel functor makes the following diagram commute.
\[\xymatrix{
\CP(\La)\ar[r]^{\mathsf{Cok}}\ar[d]^{\pi} & \mmod \La\\
\frac{\CP(\La)}{\CV}\ar[ur]_{\overline{\mathsf{Cok}}}   \\
}\]
\noindent
Here, $\pi$ is the natural quotient map, and $\CV$ is the full subcategory of $\CP(\La)$ generated by all finite direct sums of objects of the forms $(P\st{1}\rt P)$ or $(P\rt 0)$, where  $P$ is a projective module.
\end{theorem}

Note that, because of the characterization of injective objects in $\CP(\La)$, as stated in below,  the factor category $\frac{\CP(\La)}{\CV}$ in the above theorem is precisely the injectively stable category $\overline{\CP(\La)}$.

\begin{lemma}\label{proj-inde-p}\cite[Lemma 3.1]{HW}
Let $\mathsf{P}$ be an indecomposable object in the exact category $\CP(\La)$.
\begin{itemize}
\item [$(1)$]$\mathsf{P}$ is projective if and only if it is isomorphic to either  $(P\st{1}\rt P)$ or $(0 \rt P)$ for some  indecomposable projective  $\La$-module.
\item [$(2)$] $\mathsf{P}$ is injective if and only if it is isomorphic to either  $(P\st{1}\rt P)$ or $(P \rt 0)$ for some  indecomposable projective $\La$-module.
\end{itemize}
\end{lemma}

The following result shows that the functor $\mathsf{Cok}$ behaves well concerning the almost split sequences.

\begin{proposition}\label{Cok-ass}\cite[Proposition 5.6]{Ba}
Let $\eta: \mathsf{X}\st{\phi}\rightarrowtail  \mathsf{Y}\st{\psi}\twoheadrightarrow \mathsf{Z}$ be an almost split sequence in $\CP(\La)$. If  $\mathsf{Z}$ is not isomorphic to $(P\rt 0)$,  for some indecomposable projective module $P$,
 then
$$\mathsf{Cok}(\eta): \ 0 \rt \mathsf{Cok}{\rm X}\rt \mathsf{Cok}\mathsf{Y}\rt \mathsf{Cok}\mathsf{Z}\rt 0$$
is an almost split sequence in $\mmod \La$.
\end{proposition}

\begin{definition}\label{Definition-types}
An indecomposable object $\mathsf{M}$ in $\CH(\La)$ is said to be of the type $(a)$ (resp. $(b)$ and  $(c)$)  if it is    isomorphic to  $(0\rt C)$ (resp.  $(C\st{1}\rt C)$ and $(C\rt 0)$) for some indecomposable module $C$.
\end{definition}

\begin{notation}

 Let ${\mathsf{M}}$ be an object in $\CH(\La)$, and ${\mathsf{M}}=\oplus_{i \in I} {\mathsf{M}}_i$ be a decomposition of ${\mathsf{M}}$ into indecomposable objects.
We denote by  ${\mathsf{M}}_0$ the maximal direct summand of ${\mathsf{M}}$ having no direct summands isomorphic to the types $(a), (b)$, or $(c)$. For every $ * \in \{a, b, c\} $, we define  $\mathsf{M}_{*}=\oplus_{i \in I_*} \mathsf{M}_i$, where $I_* \subseteq I$ consists of all $i \in I$ such that $M_i$ is of type $*.$ By using the Krull-Schmidt property of $\CH(\La)$ we have the following unique (up to isomorphism) decomposition 
$${\mathsf{M}}={\mathsf{M}}_0\oplus {\mathsf{M}}_a\oplus {\mathsf{M}}_b\oplus {\mathsf{M}}_c$$ of ${\mathsf{M}}$.

\end{notation}

\subsection{$\tau$-tilting theory in the  module category}
We recall some notions and terminologies related to $\tau$-tilting theory from \cite{AIR}, that are necessary for this paper.

\begin{definition}\cite[Definition 0.1]{AIR}
Let $M$ be a module in $\mmod \La$.
\begin{itemize}
\item [$(1)$] We call $M$ a $\tau$-rigid module  if ${\Hom_{\La}}(M, \tau M)=0$.
\item [$(2)$]  We call $M$ a  $\tau$-tilting module if $M$ is $\tau$-rigid and $|M|=|\La|$, where $|M|$ denotes the number of indecomposable direct summand of $M$.

\item [$(3)$] We call $M$ a support $\tau$-tilting  module if there exists an idempotent $e$ of $\La$ such that $M$ is a $\tau$-tilting module in $\mmod \frac{\La}{<e>}$.
\end{itemize}
\end{definition}

 We denote by $s\tau \mbox{-}{\rm tilt}\La$ the set of isomorphism classes of  support $\tau$-tilting $\La$-modules.
\begin{definition}\cite[Definition 0.3]{AIR}
Let $(M, P)$ be a pair with $M$ in $\mmod \La$ and $P$ a projective $\La$-module.
\begin{itemize}
\item [$(1)$] We call $(M, P)$ a $\tau$-rigid pair if $M$ is $\tau$-rigid and ${\Hom}_{\La}(P, M)=0$.
\item [$(2)$]  We call $(M, P)$ a  support $\tau$-tilting object if $({\rm M}, P)$ is $\tau$-rigid and $|M|+|P|=|\La|.$
\end{itemize}
\end{definition}

 As usual, a pair $(M, P)$ is called basic if $M$ and $P$ are both basic. According to \cite[Proposition 2.3]{AIR}, the concepts of support $\tau$-tilting modules and support $\tau$-tilting pairs are essentially equivalent.

\section{ ICE-closed subcategories in $\CP(\La)$}\label{Section 3}

In this section, motivated by \cite{En} and \cite{ES2}, we introduce the notions of 
ICE-closed and IE-closed subcategories in $\CP(\Lambda)$, as well as twin objects. 
Then, in our main Theorems~\ref{Thm. ICE-rigid-Mor} and~\ref{thm:main-twin-IEclosed123}, 
we establish bijections between these subcategories, under some additional conditions, 
and tilting objects and twin objects, accordingly.

\subsection{ICE-closed subcategories in $\CP(\La)$}

Motivated by \cite{En}, we define the notion of an ICE-closed subcategory for an exact category as follows.
\begin{definition}\label{ICE}
Let $\CE$ be an exact category. A subcategory  $\CC$ of  $\CE$ is called  ICE-closed if it  satisfies the following conditions:
\begin{itemize}
    \item  $\mathscr{C}$ is  {\em extension closed}, that is,   for any conflation $X\rightarrowtail Y\twoheadrightarrow Z$ with $X,Z\in \mathscr{C}$,  $Y\in \mathscr{C}$;
    \item   $\mathscr{C}$  is  {\em admissible cokernel closed}, that is, for every admissible morphism $f:M\rt N$ with $M, N \in \CC$,  $\mathsf{Cok}f \in \CC;$
 \item   $\mathscr{C}$  is  {\em admissible image closed}, that is, for every addmisible morphism $f:M\rt N$ with $M, N \in \CC$,  $\mathsf{Im}f \in \CC.$
   \end{itemize}
\end{definition}
 ${\rm ice}~\CE$ (resp. ${\rm ice}_p\CE$) denotes the set of isomorphic classes of ICE-closed subcategories (resp.  ICE-closed subcategories with enough $\Ext$-projectives) of $\CE$. By the definition of an ICE-closed subcategory, we can observe that it is closed under finite direct sums and direct summands.

\begin{proposition}\label{Prop. ext-pro-Cok}
    Let $\CC$ be an ICE-closed subcategory of $\CP(\La).$ If  $\CC$ has enough Ext-projectives, then the following statements hold.
    \begin{itemize}
        \item [$(1)$] There exists an $\Ext$-progenarator $\mathsf{P}(\CC)$ which is rigid;
        \item [$(2)$] $\CC=\mathsf{Cok}\mathsf{P}(\CC).$
    \end{itemize}
\end{proposition}
\begin{proof}
  Since any $\Ext$-projective object $\mathsf{N}$ in $\CC$ is rigid,  by \cite[Proposition 4.14]{GNP1}, we infer  $|\mathsf{N}|\leq 2 |\La|.$ Hence, in view of assumption that  $\CC$ has  enough Ext-projectives, it  has to possess  an  $\Ext$-progenerator $\mathsf{P}(\CC).$ As $\CC$ is  admissible cokernel closed, we get $\mathsf{Cok}\mathsf{P}(\CC) \subseteq \CC$. Conversely, by assumption thus $\CC$ has enough $\Ext$-pojectives,  we have  $\CC \subseteq \mathsf{Cok}\mathsf{P}(\CC)$.
\end{proof}

\begin{lemma}\label{lem. fac-sub-add}
    Let $\mathsf{M}$ be a rigid object in $\CP(\La).$ Then, $\mathsf{Fac}\mathsf{M}\cap \mathsf{Sub}\mathsf{M}=\mathsf{add}\mathsf{M}.$
\end{lemma}
 \begin{proof}
    It is plain that $\mathsf{add}\mathsf{M} \subseteq \mathsf{Fac}\mathsf{M}\cap \mathsf{Sub}\mathsf{M}$.   Conversely, let $\mathsf{N} \in \mathsf{Fac}\mathsf{M}\cap \mathsf{Sub}\mathsf{M}$. Take a left $\mathsf{add}\mathsf{M}$-approximation $\phi \colon  \mathsf{N}\to \mathsf{M}'$ with $\mathsf{M}'\in \mathsf{add}\mathsf{M}$. Since $ \mathsf{N} \in  \mathsf{Sub}\mathsf{M}$, $\phi$ factors through an inflation in $\CP$. This follows that $\phi$ is an inflation in $\CH.$ Then, we have the following conflation
    $ \epsilon: \mathsf{N}\st{\phi}\rightarrowtail \mathsf{M}'\twoheadrightarrow \mathsf{K}$ in $\CH$. By applying $\Hom_{\CH}(-,\mathsf{M})$, we obtain the following exact sequence 
  \[
  \begin{tikzcd}
    \Hom_{\CH}(\mathsf{M}', \mathsf{M}) \rar["{(-)\circ\phi}"] &  \Hom_{\CH}(\mathsf{N}, \mathsf{M}) \rar & \Ext^1_{\CH}(\mathsf{K}, \mathsf{M}) \rar &  \Ext^1_{\CH}(\mathsf{M}', \mathsf{M}).
  \end{tikzcd}
  \]
  Since $\phi$ is a left $\mathsf{add}\mathsf{M}$-approximation, $(-)\circ\varphi$ is a surjection.
  In addition, $\Ext_{\CH}(\mathsf{M}', \mathsf{M})=0$,  since $\mathsf{M}$ is rigid.  Hence $\Ext_{\CH}^1(\mathsf{K}, \mathsf{M}) =  0$. As $\mathsf{N}$ belongs to $\mathsf{Fac}\mathsf{M}$, then there exists a conflation $\la: \mathsf{D} \rightarrowtail \mathsf{M}''\twoheadrightarrow \mathsf{N}$ in $\CP$ with $\mathsf{M}'' \in \mathsf{add}\mathsf{M}$. As $\mathsf{D}$   has projective dimension at most one in $\CH$, we have the following exact sequence by using the conflation $\la$
  \[
  \begin{tikzcd}
    \Ext_{\CH}^1(\mathsf{K}, \mathsf{M}'') \rar & \Ext_{\CH}^1(\mathsf{K}, \mathsf{N}) \rar & 0.
  \end{tikzcd}
  \]
  Since we have $\Ext_{\CH}^1(\mathsf{K}, \mathsf{M}'') = 0$, we obtain $\Ext_{\CH}^1(\mathsf{K}, \mathsf{N}) = 0$. It follows that the conflation $\epsilon$ splits, which implies $\mathsf{N} \in \mathsf{add}\mathsf{M}$.
 \end{proof}
\begin{proposition}\label{Prop. Cok.Ice-subcatgory}
Let $\mathsf{M}$ be a basic rigid object in $\CP(\La).$
\begin{itemize}
    \item [$(1)$] The following conditions are equivalent for $\mathsf{X}$ in $\CP(\La)$;
    \begin{itemize}
        \item [$(a)$] $\mathsf{X} \in \mathsf{Cok}\mathsf{M},$
        \item [$(b)$] There exists a conflation
        $$ \mathsf{M}_1\rightarrowtail \mathsf{M}_0\twoheadrightarrow \mathsf{X},$$
        where $\mathsf{M}_0, \mathsf{M}_1 \in \mathsf{add}\mathsf{M}.$
    \end{itemize}
    \item [$(2)$] $\mathsf{Cok}\mathsf{M}$ is an ICE-closed subcategory of $\CP(\La);$
    \item [$(3)$] $\mathsf{P}(\mathsf{Cok}\mathsf{M})\simeq \mathsf{M}.$
\end{itemize}
\end{proposition}
\begin{proof}
    $(1)$ It is evident  that $(b)$ implies that $(a)$. Let $\mathsf{X}$ belong to $\mathsf{Cok}\mathsf{M},$ then we have $\mathsf{M}_1\st{f}\rt \mathsf{M}_0\twoheadrightarrow \mathsf{X}$, where $\mathsf{M}_0, \mathsf{M}_1 \in \mathsf{add}\mathsf{M}.$ Clearly, $\mathsf{Im} f \in \mathsf{Fac}\mathsf{M}\cap \mathsf{Sub}\mathsf{M}$. Hence, according to Lemma \ref{lem. fac-sub-add}, $\mathsf{Im} f \in \mathsf{add}\mathsf{M},$  leading to get the  desired conflation in $(b)$.
$(2)$ We first show that $\mathsf{Cok}\mathsf{M}$ is closed under extension. To this end, we take a conflation $\la: \mathsf{N}_1 \rightarrowtail \mathsf{N}_3\twoheadrightarrow \mathsf{N}_2$, where $\mathsf{N}_1, \mathsf{N}_2 \in \mathsf{Cok}\mathsf{M}$. Then, for $i=1, 2$, there exist conflations $\la_i: \mathsf{M}^i_1\rightarrowtail \mathsf{M}^i_0\twoheadrightarrow \mathsf{N}_i$, where $\mathsf{M}^i_1, \mathsf{M}^i_0 \in \mathsf{add}\mathsf{M}.$ In particular, for the case $\la_1$, applying the facts that $\CP$ is a hereditary exact category and $\mathsf{M}$ is a rigid object, we can show that the conflation $\la$ provides an exact sequence of abelian groups when we apply the functor $\Hom_{\CP}(\mathsf{M}, -)$ on it. This enables us to apply the horseshoe lemma, leading to a conflation  $\mathsf{M}^0_1\oplus  \mathsf{M}^0_1 \rightarrowtail \mathsf{M}^1_0 \oplus  \mathsf{M}^0_0\twoheadrightarrow \mathsf{N}_3$, so proving the claim. Now, we show that $\mathsf{Cok}\mathsf{M}$ is closed under admissible cokernels. Let $f:\mathsf{N}\rt \mathsf{K}$ be an admissible morphism in $\CP$ and both domain and codomain in $\mathsf{Cok}\mathsf{M}.$ Hence, $\mathsf{Im}f \in \CP$, and we have the factorization of $f$ through deflation $p:\mathsf{N}\twoheadrightarrow \mathsf{Im} f$ and inflation $i: \mathsf{Im} f\rightarrowtail \mathsf{K}.$ Since $ \mathsf{K} \in \mathsf{Cok}\mathsf{M}$, there is a conflation $ \mathsf{M}_1\rightarrowtail \mathsf{M}_0\twoheadrightarrow  \mathsf{K}$, where $\mathsf{M}_0, \mathsf{M}_1 \in \mathsf{add}\mathsf{M}.$ By taking pullback, we obtain the following diagram

  \begin{equation} \label{diag}
\xymatrix{\\	 & \mathsf{M}_1 \ar@{>->}[d] \ar@{=}[r] & \mathsf{M}_1 \ar@{>->}[d]
 &   & \\
	 & \mathsf{Q} \ar@{->>}[d] \ar@{>->}[r] & \mathsf{M}_0
	\ar@{->>}[d] \ar@{->>}[r] &   \mathsf{Cok} f \ar@{=}[d]  & \\
  & \mathsf{Im} f  \ar@{>->}^{i}[r] & \mathsf{K}
	 \ar@{->>}[r] & \mathsf{Cok} f    & \\
&  &   &  & }
\end{equation}
Taking pullback along the morphism $p: \mathsf{N}\rt \mathsf{Im}f$ and the conflation lying in the first column of the above diagram gives us the following diagram:

$$\xymatrix{
	 & \mathsf{M}_1 \ar@{=}[d] \ar@{>->}[r] & \mathsf{L}
	\ar@{->>}[d] \ar@{->>}[r] &   \mathsf{N} \ar@{->>}^p[d]  & \\
  & \mathsf{M}_1  \ar@{>->}[r] & \mathsf{Q}
	 \ar@{->>}[r] & \mathsf{Im} f  & \\
&   &   &  & }
$$
We have already proved that $\mathsf{Cok}\mathsf{M}$ is closed under extension. Hence,  the upper conflation in the above diagram implies that  $\mathsf{L}$ lies in $\mathsf{Cok}\mathsf{M}$. Consequently,  $\mathsf{Q} \in \mathsf{Fac}\mathsf{M}\cap \mathsf{Sub}\mathsf{M} $, and according to Lemma \ref{lem. fac-sub-add}, $\mathsf{Q} \in \mathsf{add}\mathsf{M}$. Hence, the second row of diagram \ref{diag} provides us a conflation such that the first two terms belong to $\mathsf{add}\mathsf{M}$. This completes the claim. To show that $\mathsf{Cok}\mathsf{M}$ is closed under images of admissible morphisms, we need to make a simple modification of the counterpart of \cite[Proposition 3.5(2)]{En}, which we leave to the reader.

$(3)$ Suppose $\mathsf{X}$ is any object in $\mathsf{Cok}\mathsf{M}$. According to $(1)$, there is a conflation $ \mathsf{M}_1\rightarrowtail \mathsf{M}_0\twoheadrightarrow \mathsf{X}$, where $\mathsf{M}_0$ and $\mathsf{M}_1$ are in $\mathsf{add}\mathsf{M}$. Using the long exact sequences and $\CP$ being a hereditary exact category, we find that $\mathsf{M}$ is an $\Ext$-projective object in $\mathsf{Cok}\mathsf{M}$. Thanks to Proposition \ref{Prop. ext-pro-Cok}, there exists the Ext-progenrator  $\mathsf{P}(\mathsf{Cok}\mathsf{M})$. In addition, from $(1)$, there is a conflation $\mathsf{M}_1\rightarrowtail \mathsf{M}_0\twoheadrightarrow \mathsf{P}(\mathsf{Cok}\mathsf{M})$ with  $\mathsf{M}_0$ and $\mathsf{M}_1$ are in $\mathsf{add}\mathsf{M}$. As $\mathsf{P}(\mathsf{Cok}\mathsf{M})$ is $\Ext$-projective, the sequence splits,  and because of this fact that both $\mathsf{M}$ and $\mathsf{P}(\mathsf{Cok}\mathsf{M})$ are basic,  we obtain the desired isomorphism in $(3)$.
\end{proof}
Now we are ready to state Enomoto's bijection for the hereditary exact category $\CP(\La)$.
\begin{theorem}\label{Thm. ICE-rigid-Mor}
  Let $\Lambda$ be an Artin algebra. Then we have the following bijections
  \[
  \begin{tikzcd}
    {\rm rigid} \CP(\La) \rar[shift left, "\mathsf{Cok}"] & {{\rm ice}_p\CP(\La)}\lar[shift left, "\mathsf{P}"]
  \end{tikzcd}   \]
  \end{theorem}
\begin{proof}
    Propositions \ref{Prop. ext-pro-Cok} and \ref{Prop. Cok.Ice-subcatgory} establish that the maps $\mathsf{Cok}$ and $\mathsf{P}$ are well-defined. Furthermore, these propositions also show that they are mutually inverse to each other.
\end{proof}

\subsection{IKE-closed subcategories in $\CP(\La)$}
Here, we define the dual notion of an ICE-closed subcategory introduced in Definition~\ref{ICE}.

\begin{definition}\label{IKE}
Let $\CE$ be an exact category. A subcategory  $\CC$ of  $\CE$ is called  an IKE-closed if it is closed under  admissible image, admissible kernel and extension. Here, we call a subcategory is closed under admissible kernel if   for every admissible morphism $f:M\rt N$ in the subcategory, then $\mathsf{Ker}f$ belongs to the subcategory.  We denote by  ${\rm ike} \CE$ (resp. ${\rm ike}_i \CE$) the set of equivalent class of IKE-closed subcategories (resp. IKE-closed subcategories with enough $\Ext$-injectives)  of $\CE$.
\end{definition}

Applying the duality \( (-)^* = \Hom_{\Lambda}(-,\Lambda) \) componentwise to objects in \(\CP(\Lambda)\), we obtain a duality between \(\CP(\Lambda)\) and \(\CP(\Lambda^{\mathrm{op}})\). We use the same notation \( (-)^* \) for this induced duality. By applying the duality \( (-)^* \) to the bijection in Theorem \ref{Thm. ICE-rigid-Mor}, now for the opposite Artin algebra \(\Lambda^{\mathrm{op}}\), we obtain:

\begin{itemize}
\item \({\rm rigid}\,\CP(\Lambda^{\mathrm{op}})\) is sent to \({\rm rigid}\,\CP(\Lambda)\);
\item the subcategories \({\rm ice}_p\,\CP(\Lambda^{\mathrm{op}})\) are sent to the subcategories \({\rm ike}_i\,\CP(\Lambda)\);
\item the bijections \(\mathsf{Cok}\) and \(\mathsf{P}\) correspond respectively to \(\mathsf{Ker}\) and \(\mathsf{I}\).
\end{itemize}

Combining all these facts, we obtain the following commutative diagram of bijections, where the vertical bijections are induced by the duality \( (-)^* \):

\[
\begin{tikzcd}[column sep = large]
{\rm rigid}\,\CP(\Lambda^{\mathrm{op}}) \dar[shift left, "(-)^*"] \rar["{\mathsf{Cok}}", shift left]
  & {\rm ice}_p\,\CP(\Lambda^{\mathrm{op}}) \dar[shift left, "(-)^{*}"] \lar["{\mathsf{P} }", shift left]\\
{\rm rigid}\,\CP(\Lambda) \rar["{\mathsf{Ker}}", shift left] \uar["(-)^*", shift left]
  & {\rm ike}_i\,\CP(\Lambda^{\mathrm{op}}) \lar["{\rm I}", shift left] \uar["(-)^{*}", shift left]
\end{tikzcd}
\]

According to above discussion, we have the following bijection between rigid objects in $\CP(\La)$ and another type of the subcategories.
\begin{theorem}\label{Thm. IKE-rigid-Mor}
  Let $\Lambda$ be an Artin algebra. Then, we have the following bijections
  \[
  \begin{tikzcd}
    {\rm rigid} \CP(\La) \rar[shift left, "\mathsf{Ker}"] & {{\rm ike}_i\CP(\La)}\lar[shift left, "\mathsf{I}"]
  \end{tikzcd}   \]
  \end{theorem}
\subsection{Twin objects in $\CP(\La)$}

Motivated by \cite[Definition 1.1]{ES2}, where twin modules are introduced, we define twin objects in $\CP(\La)$ analogously.

\begin{definition}\label{def:twin-rigid}
  Let $\mathsf{P}$ and $\mathsf{I}$ be two objects in $\CP(\La)$. We say that a pair $(\mathsf{P},\mathsf{I})$ is a \emph{twin rigid object} if the following conditions are satisfied.
  \begin{enumerate}
      \item $\mathsf{P}$ and $\mathsf{I}$ are rigid, that is, $\Ext_{\CP}^{1}(\mathsf{P},\mathsf{P})=0$ and $\Ext_{\CP}^{1}(\mathsf{I}, \mathsf{I})=0$ hold.
      \item There are two conflations
    \begin{equation}\label{eq:twin-rigid-1}     
\mathsf{P}  \rightarrowtail   \mathsf{I}^0 \twoheadrightarrow \mathsf{I}^1
    \end{equation}
    \begin{equation}\label{eq:twin-rigid-2}
    \mathsf{P}_1 \rightarrowtail \mathsf{P}_0 \twoheadrightarrow \mathsf{I}
    \end{equation}
  with $\mathsf{I}^{0}, \mathsf{I}^{1}\in\mathsf{add}~\mathsf{I}$ and $\mathsf{P}_{0}, \mathsf{P}_{1}\in\mathsf{add}~\mathsf{P}$.
  \end{enumerate}
  \end{definition}
A twin pair object $(\mathsf{P},\mathsf{I})$ is called basic if both $\mathsf{P}$ and $\mathsf{I}$ are basic. Similar to \cite[Example 1.1]{ES2} for the twin pair objects in the module category, we have the following three extreme types of examples of twin pair objects in the morphism category. 
\begin{example}

  \begin{itemize}
    \item $(\mathsf{T}, (\La\st{1}\rt \La)\oplus(\La\rt 0))$ is twin rigid if and only if $\mathsf{T}$ is tilting.
    \item $((\La\st{1}\rt \La)\oplus (0 \rt \La), \mathsf{U})$ is twin rigid if and only if $\mathsf{U}$ is cotilting.
    \item $(\mathsf{P},\mathsf{P})$ is twin rigid if and only if $\mathsf{P}$ is rigid.
  \end{itemize}
\end{example}
According to this example in two ways we can send tilting objects in $\CP(\La)$ into twin rigid objects. Here, we should mention that tilting objects and cotiliting objects coincide in $\CP(\La),$ see \cite[Theorem 4.3]{GNP1}.

\begin{definition}\label{IE1}
Let $\CE$ be an exact category. A subcategory  $\CC$ of  $\CE$ is called  IE-closed if it is  {\em admissible image closed} and {\em extension closed}.
\end{definition}

\begin{proposition}\label{prop:eic11}
Let $(\mathsf{P}, \mathsf{I})$ be a twin rigid object.
Set $\CC = \mathsf{Cok} \mathsf{P} \cap \mathsf{Ker}\mathsf{I}$. Then the following statements hold.
\begin{enumerate}
  \item $\CC$ is an IE-closed  subcategory of $\CP(\La)$;
    \item $\CC$ has an $\Ext$-progenerator $\mathsf{P}$ and an $\Ext$-procogenerator $\mathsf{I}$.
\end{enumerate}
\end{proposition}

\begin{proof}
$(1)$ According to Theorems \ref{Thm. ICE-rigid-Mor} and \ref{Thm. IKE-rigid-Mor},  the subcategories $\mathsf{Cok} \mathsf{P}$ and $\mathsf{Ker}\mathsf{I}$ are ICE-closed and IKE-closed respectively. Hence, the subcategory $\CC$ has to be  IE-closed. 

$(2)$  We only prove the statement for $\mathsf{P}.$ From the conflation \ref{eq:twin-rigid-1} in Definition \ref{def:twin-rigid}, we have $\mathsf{P} \in \mathsf{Ker}\mathsf{I}$. So, $\mathsf{P} \in \CC$.  On the other hand, from  Proposition \ref{Prop. Cok.Ice-subcatgory}, we know that $\mathsf{P}$ is an $\Ext$-progenerator for  $\mathsf{Cok} \mathsf{P}$. Let $\mathsf{X}$ be an arbitrary object in $\CC$. There exists a conflation in $\mathsf{Cok} \mathsf{P}$ 
 \begin{equation}\label{eq. conf. Ext.IE}   
\mathsf{K} \rightarrowtail \mathsf{P}_0 \twoheadrightarrow \mathsf{X}
    \end{equation}
with $\mathsf{P}_0 \in \mathsf{add} \mathsf{P}.$ Since  $\mathsf{Ker}\mathsf{I}$ is closed under kernel, we obtain $\mathsf{K} \in \mathsf{Ker}\mathsf{I}$. Hence, the conflation lies in $\CC$ as well. Hence,  $\mathsf{P}$ is also an $\Ext$-progenerator for $\CC$. 
\end{proof}
\begin{proposition}\label{prop:eic_pair2}
Let $\CC$ be a IE-closed subcategory of $\CP(\La)$. Assume that $\CC$ has an $\Ext$-progenerator $\mathsf{P}$ and an $\Ext$-progcoenerator $\mathsf{I}$. Then, the following statements hold.
\begin{enumerate}
    \item $(\mathsf{P}, \mathsf{I})$ is a twin rigid object; 
  \item $\CC = \mathsf{Cok}\mathsf{P} \cap \mathsf{Ker}\mathsf{I}$ holds.
\end{enumerate}
\end{proposition}

\begin{proof}  
$(1)$ By definition, $\mathsf{P}$ and $\mathsf{I}$ are rigid. Since $\mathsf{P}$ is an $\Ext$-injective cogenerator of $\CC$, there exists an exact sequence
\begin{equation*} 
\mathsf{P}_1\st{f}\rightarrowtail \mathsf{P}_0 \twoheadrightarrow \mathsf{I}
    \end{equation*}
with $\mathsf{P}_0, \mathsf{P}_1\in\mathsf{add}\mathsf{P}$. By Lemma~\ref{lem. fac-sub-add}, we have $\mathsf{Im} f\in\mathsf{add} \mathsf{P}$, and hence we obtain the sequence \eqref{eq:twin-rigid-2}. By the dual argument, we obtain the sequence \eqref{eq:twin-rigid-1}.
$(2)$
By use of Proposition \ref{Prop. Cok.Ice-subcatgory} and its dual, every object in $\CC$ is the image of some map $\mathsf{P}' \to \mathsf{I}'$ with $\mathsf{P}' \in \mathsf{add} \mathsf{P}$ and $\mathsf{I}' \in \mathsf{add}\mathsf{I}$. Since $\mathsf{P}', \mathsf{I}' \in \CC$ and $\CC$ is closed under images, we obtain $\mathsf{Cok}\mathsf{P} \cap \mathsf{Ker}\mathsf{I} \subseteq \CC$.
Conversely, every object in $\CC$ belongs to both $\mathsf{Cok}\mathsf{P}$ and $\mathsf{Ker}\mathsf{I}$ since $\mathsf{P}$ is an $\Ext$-progenerator and $\mathsf{I}$ is an $\Ext$-injective cogenerator of $\CC$. Thus $\CC =\mathsf{Cok}\mathsf{P} \cap \mathsf{Ker}\mathsf{I} $ holds.
\end{proof}

\begin{theorem}\label{thm:main-twin-IEclosed123}
  The assignments $\CC \mapsto (\mathsf{P}(\CC), \mathsf{I}(\CC))$ and $(\mathsf{P}, \mathsf{I}) \mapsto \mathsf{Cok} \mathsf{P} \cap \mathsf{Ker} \mathsf{I}$ give bijections between the following two sets:
  \begin{enumerate}
    \item The set of  equivalent classes of IE-closed subcategories $\CC$ of $\CP(\La)$ which  admit an $\Ext$-progenerator  and an $\Ext$-injective cogenerator;
    \item The set of isomorphism classes of basic twin rigid objects $(\mathsf{P}, \mathsf{I})$.
  \end{enumerate}
\end{theorem}
\begin{proof}
This follows immediately from Propositions \ref{prop:eic11} and \ref{prop:eic_pair2}.
\end{proof}

\section{Cotorsion-torsion triples}\label{Section 4}
In this section, we prove the existence of a bijection between cotorsion--torsion triples and tilting objects in $\CP(\Lambda)$. In Proposition~\ref{prop. mono-tor-free}, we show that the morphisms in the right-most subcategory of a cotorsion--torsion triple are monomorphisms. A dual version of our results is also considered.

Let us introduce the following notation to state our next lemma.

\begin{notation}
 Let $\CC$ and $\mathscr{D}$ be subcategories of $\CP(\La)$. We denote by $\mathsf{Cone}(\CC, \mathscr{D})$, the full subcategory of $\CP$ whose objects are those $\mathsf{X}$ for which there is a conflation
 $$ \mathsf{M}_1\rightarrowtail \mathsf{M}_0\twoheadrightarrow \mathsf{X}, $$
where $\mathsf{M}_0 \in \CC,  \mathsf{M}_1 \in \mathscr{D}.$ Dually, we denote by $\mathsf{Fib}(\CC, \mathscr{D})$, the full subcategory of $\CP(\La)$ whose objects $\mathsf{Y}$ admit a conflation $ \mathsf{Y}\rightarrowtail \mathsf{M}_0\twoheadrightarrow \mathsf{M}_1,$ where $\mathsf{M}_0 \in \CC, \mathsf{M}_1 \in \mathscr{D}.$ For an object $\mathsf{M}$ in $\CP(\La)$, we denote $\mathsf{Cone}(\mathsf{M}, \mathsf{M})$ and $\mathsf{Fib}(\mathsf{M}, \mathsf{M})$ respectively when $\CC=\mathscr{D}=\mathsf{addM}$.
\end{notation}

\begin{lemma}\label{Lem. tilt.cotilt.prep}
    Let $\mathsf{T}$ be a rigid object in $\CP(\La).$ Then, the following conditions are equivalent.
    \begin{itemize}
        \item [$(i)$] $\mathsf{T}$ is a tilting object;
        \item [$(ii)$]  $\mathsf{T}^{\bot}=\mathsf{Cone}(\mathsf{T}, \mathsf{T}); $
        \item [$(iii)$]  ${}^{\bot}\mathsf{T}=\mathsf{Fib}(\mathsf{T}, \mathsf{T}). $
    \end{itemize}
\end{lemma}

\begin{proof}
The  result is proved by the items $(iii), (vi)$ and $(vi)'$ of \cite[Theorem 4.3]{GNP1}.
\end{proof}
\begin{definition}
   Let $\CE = (\CA, \CS)$ be an exact category with enough projectives.
A triple $(\CC, \DD, \CF)$ is called a \emph{cotorsion--torsion triple} if $(\CC, \DD)$ is a  (complete) cotorsion pair and $(\DD, \CF)$ is a torsion pair.  
Dually, a \emph{torsion--cotorsion triple} is a triple $(\CC, \DD, \CF)$ such that $(\CC, \DD)$ is a torsion pair and $(\DD, \CF)$ is a cotorsion pair.  
Since the first and third subcategories of such a triple are uniquely determined by the middle subcategory, we may identify a triple with its middle subcategory.
\end{definition}
The  definitions of  cotorsion pairs and  torsion pairs in an exact category are defined in a similar way  for abelian categories, see \cite[Definitions 2.7 and 2.8]{PZZ}.

The following theorem establishes  a bijection between tilting subcategories and cotorsion-torsion triples.  For a proof we refer to \cite[Theorem 3.8 ]{PZZ} or \cite[Theorem 5.2]{GNP1}.

\begin{theorem}\label{bijections}
	Assume $\mathscr{E}$ is weakly idempotent complete. Then there are mutually inverse bijections:
	\begin{align*}
		\{\text{tilting subcategories}\} & \leftrightarrow\{\text{cotrsion-torsion triples}\}\\
		\mathscr{T}& \mapsto (^{\bot}\mathsf{Fac}\mathscr{T},\mathsf{Fac}\mathscr{T}, \mathsf{Fac}\mathscr{T}^{\perp_0})\\
		\mathscr{C} \cap \mathscr{D}& \mapsfrom (\mathscr{C},\mathscr{D}, \CF).
	\end{align*}
    
\end{theorem}

\begin{proposition}\label{pro. cot-tor.triple}
If $\mathsf{T}$ is a tilting object in $\CP(\La)$, then the following statements holds: 
\begin{itemize}
\item [$(1)$] $\mathsf{Cok}\mathsf{T}=\mathsf{Fac}\mathsf{T}$;
\item [$(2)$] $\mathsf{Cok}\mathsf{T}$ is a torsion subcategory; 
\item [$(3)$] $({}^{\perp}(\mathsf{Cok}\mathsf{T}),~~\mathsf{Cok}\mathsf{T})$ is a (complete) cotorsion pair.
\end{itemize}
In particular, the twin pair $(\mathsf{T}, (\La\st{1}\rt \La)\oplus (\La\rt 0) )$ is sent under the bijection in Theorem \ref{thm:main-twin-IEclosed123} to the cotrosion-torsion triple $({}^{\perp}(\mathsf{Cok}\mathsf{T}),~~\mathsf{Cok}\mathsf{T}, ~~(\mathsf{Cok}\mathsf{T})^{\perp_0})$
\end{proposition}
\begin{proof}
$(1)$  It is easy to see we have the following chain of subcategories:
\[\mathsf{Cone}(\mathsf{T}, \mathsf{T}) \subseteq \mathsf{Cok}\mathsf{T} \subseteq \mathsf{Fac}\mathsf{T} \subseteq \mathsf{T}^{\perp}. \]
By Lemma \ref{Lem. tilt.cotilt.prep}, we obtain $\mathsf{Cone}(\mathsf{T}, \mathsf{T})=\mathsf{T}^{\perp}.$
Hence, the identity in $(1)$ is proved.  $(2)$ and $(3)$  follow from Theorem \ref{bijections} in conjunction with $(1)$. The last statement follows from the fact that $\mathsf{Ker}(\La\st{1}\rt \La)\oplus (\La\rt 0) )=\CP(\La).$
\end{proof}

In the next result, we provide further partial information on the third subcategory in the cotorsion-torsion triple $({}^{\perp}\mathsf{Cok}\mathsf{T},\; \mathsf{Cok}\mathsf{T},\; \mathsf{Cok}\mathsf{T}^{\perp_0}).$

\begin{proposition}\label{prop. mono-tor-free}
With the same notation as in Proposition~\ref{pro. cot-tor.triple}, if $(P_1 \xrightarrow{\,f\,} P_0)$ is an object of\/ $\mathsf{Cok}\mathsf{T}^{\perp_0}$, then $ \mathsf{Ker} f=0$. In particular, the projective dimension of $\mathsf{Cok} f$ is at most one. 
\end{proposition}

\begin{proof}
    Let $K := \mathsf{Ker} f$, and let $P \xrightarrow{g} K \to 0$ be an epimorphism. 
Then we have a morphism in $\CP(\Lambda)$
\[
(ig,\,0)\colon (P \to 0) \to (P_1 \xrightarrow{f} P_0),
\]
where $i\colon K \hookrightarrow P_1$ is the inclusion.  
However, the injective object $(P \to 0)$ lies in $\mathsf{Cok}\,\mathsf{T}$. 
This follows from the fact that $\mathsf{Cok}\,\mathsf{T}$ is the left component of a cotorsion pair.  
Since $(\mathsf{Cok}\,\mathsf{T},\,\mathsf{Cok}\,\mathsf{T}^{\perp_0})$ is a torsion pair, the morphism $(ig,\,0)$ must be zero.  
Hence $g=0$, which implies that $K=0$, as desired.    
\end{proof} 

\begin{proposition}\label{prop. bij-twin-tilting-cotortor}
There exist mutually
inverse bijections between the following sets.
\begin{itemize}
\item [$(1)$] Isomorphism classes of basic twin objects with second component to be $(0\rt \La)\oplus (\La\st{1}\rt \La)$ in $\CP(\La)$;
\item [$(2)$] Isomorphism classes of basic tilting objects  in $\CP(\La)$; 
\item [$(3)$] Cotorsion-torsion triples.
\end{itemize}
The maps from $(1)$ to $(2)$  and the converse  are only given by projection  and injection, respectively.   The map from $(2)$ to $(3)$ is given by $\mathsf{T}\mapsto ({}^{\perp}(\mathsf{Cok}\mathsf{T}),\; \mathsf{Cok}\mathsf{T},\; (\mathsf{Cok}\mathsf{T})^{\perp_0})$, and the map from $(3)$ to $(2)$ is given by $ (\CC, \DD, \CF) \mapsto \mathsf{P}(\CC \cap \CF)$.  
\end{proposition}

\begin{proof}
    The maps between $(1)$ and $(2)$ are clearly mutually bijective. 
The mutual bijection between $(2)$ and $(3)$ follows from 
Theorem~\ref{bijections}. 
Here, we note that for a cotorsion--torsion triple $(\CC, \DD, \CF)$ in $\CP(\La)$, 
we have $
\CC \cap \CF = \mathsf{add}\,\mathsf{P}(\CC \cap \CF). $
Hence, we may identify tilting subcategories with basic tilting objects.
\end{proof}

The duality $(-)^* \colon \CP(\La) \to \CP(\La^{\rm op})$ yields dual statements of 
Propositions~\ref{pro. cot-tor.triple} and~\ref{prop. bij-twin-tilting-cotortor}.

\begin{proposition}\label{prop. ker-torsion-cotrosion}
Let $\mathsf{T}$ be a tilting object in $\CP(\La)$. Then the following statements hold: 
\begin{itemize}
\item [$(1)$] $\mathsf{Ker}\mathsf{T}=\mathsf{Sub}\mathsf{T}$, and $\mathsf{Ker}\mathsf{T}$ is a torsion-free subcategory; 
\item [$(2)$] $(\mathsf{Ker}\mathsf{T},~~(\mathsf{Ker}\mathsf{T})^{\perp})$ is a (complete) cotorsion pair.
\end{itemize}
In particular, the twin pair $((\La\st{1}\rt \La)\oplus (0\rt \La), \mathsf{T})$ is sent under the bijection in Theorem \ref{thm:main-twin-IEclosed123} to the torsion-cotorsion triple $({}^{\perp_0}(\mathsf{Ker}\mathsf{T}), ~~\mathsf{Ker}\mathsf{T},~~(\mathsf{Ker}\mathsf{T})^{\perp})$
\end{proposition}
Let $\CC$ be a subcategory in $\CP(\La)$. We denote by $\CC^*$ the image of $\CC$ in $\CP(\La^*)$  under the duality $(-)^*$. 

\begin{proposition}\label{prop. bij-twin-tilting-cotortor}
There exist mutually
inverse bijections between the following sets.
\begin{itemize}
\item [$(1)$] Isomorphism classes of basic twin objects with first component to be $(\La \rt 0)\oplus (\La\st{1}\rt \La)$ in $\CP(\La)$;
\item [$(2)$] Isomorphism classes of basic tilting objects  in $\CP(\La)$; 
\item [$(3)$] Torsion-cotorsion triples.
\end{itemize}
The maps from $(1)$ to $(2)$  and the converse  are only given by projection  and injection, respectively.   The map from $(2)$ to $(3)$ is given by $\mathsf{T}\mapsto ({}^{\perp_0}(\mathsf{Ker}\mathsf{T}), ~~\mathsf{Ker}\mathsf{T},~~(\mathsf{Ker}\mathsf{T})^{\perp})$, and the map from $(3)$ to $(2)$ is given by the composition:
$$ (\CC, \DD, \CF) \mapsto (\CF^*, \DD^*, \CC^*) \mapsto (\mathsf{P}(\CF^* \cap \CC^*))^*.$$  
\end{proposition}

Here, we provide a dual version of  Proposition~\ref{prop. mono-tor-free} formulated in terms of the transpose functor.

\begin{proposition}
    With the same notation as in Proposition~\ref{prop. ker-torsion-cotrosion}, if $(P_1 \xrightarrow{\,f\,} P_0)$ is an object of~${}^{\perp_0}(\mathsf{Ker}\mathsf{T})$, then projective dimension of $\mathsf{Tr}(\mathsf{Cok}f)$ is at most one.
    \end{proposition}
\begin{proof}
  Let $(P_1 \xrightarrow{\,f\,} P_0)$ be an object of~${}^{\perp_0}(\mathsf{Ker}\mathsf{T})$. Applying the duality $(-)^* \colon \CP(\La) \to \CP(\La^{\rm op})$, we obtain $((P_0)^*\st{f^*}\rt (P_1)^*)$ as an object of $({}^{\perp_0}(\mathsf{Ker}\mathsf{T}))^*=(\mathsf{Cok}\mathsf{T}^*)^{\perp_0}$. Hence, by Proposition \ref{prop. mono-tor-free}, $f^*$ is injective. It follows that $\mathsf{Cok}f^*$ has dimension at most one. Since $\mathsf{Tr}(\mathsf{Cok}f)$ is a direct summand of $\mathsf{Cok}f^*$, this completes the proof.  
\end{proof}

\section{$P$-right exact sequences}

In this section, we introduce the notion of $P$-right exact sequences, which is the main ingredient for the definition of rICE-closed subcategories, which will be defined in the next section. Over hereditary algebras, we provide a criterion to determine $P$-right exact sequences.

\vspace{.3 cm}

\subsection{$P$-right exact sequences}
Let $(A_3, J )$ be the quiver $A_3:v_2\st{a}\rt v_1\st{b}\rt v_0$  with relation $J$ generated by $ab$. By definition, a representation $M$ of $(A_3, J)$ over $\mmod \La$ is a diagram

$$M:M_2\st{f_2}\rt M_1\st{f_1}\rt M_1$$
of $\La$-modules and $\La$-homorphisms with $f_1f_2=0.$ A morphism between the representations $M$ and $N$ is a triple $\alpha=(\alpha_2, \alpha_1, \alpha_0)$ of $\La$-homorphisms such that the diagram

\[
\xymatrix{	  M_2 \ar[d]^{\alpha_2} \ar[r]^{f_2} &  M_1 \ar[d]^{\alpha_1}
\ar[r]^{f_1} &  M_0\ar[d]^{\alpha_0} & \\
	 N_2 \ar[r]^{g_2} & N_1
	 \ar[r]^{g_1} & N_0   & }
\]
is commutative. Denote by ${\rm rep}(A_3, J, \La)$ the category of all representations of $(A_3, J )$ over $\mmod \La.$

It is not difficult to see that indecomposable projective objects of ${\rm rep}(A_3, J, \La)$ are either $(P\st{1}\rt P\rt 0), (0 \rt P\st{1}\rt P)$ or $(0\rt 0\rt P)$, where $P$ is an indecomposable projective $\La$-module.  Let $\la: L \st{f }\rt M\st{g} \rt N \rt  0$ be a right exact sequence in $\mmod \La.$  We denote by  $\Omega^2\la:(L_2\st{\Omega^2f}\rt M_2\st{\Omega^2g}\rt N_2)$  the second syzygy of $\la$ in ${\rm rep}(A_3, J, \Lambda)$.

\hspace{ 1cm}

 The following construction is crucial in the rest of the paper.

\begin{construction}\label{Construction5}
    Let $\la: L \st{f }\rt M\st{g} \rt N \rt  0$ be a right exact sequence in $\mmod \La.$ Let $K:=\mathsf{Ker} g=\mathsf{Im} f.$ By using the horseshoe lemma, we get the following commutative diagram:

\[\begin{tikzcd}[column sep=10pt]
	&& 0 & 0 & 0 \\
	0 & 0 & {\Omega_{\Lambda}(K)\oplus Q'} & {\Omega_{\Lambda}(M)\oplus Q} & {\Omega_{\Lambda}(N)} & 0 \\
	{\Omega_{\Lambda}(L)} & 0 & {P_L} & {P_L\oplus P_N} & {P_N} & 0 \\
	{P_L} & 0 & K & M & N & 0 \\
	L && 0 & 0 & 0 \\
	0
	\arrow[draw=none, from=1-3, to=1-4]
	\arrow[from=1-3, to=2-3]
	\arrow[draw=none, from=1-4, to=1-5]
	\arrow[from=1-4, to=2-4]
	\arrow[from=1-5, to=2-5]
	\arrow[from=2-1, to=3-1]
	\arrow[from=2-2, to=2-3]
	\arrow["\psi", from=2-3, to=2-4]
	\arrow[from=2-3, to=3-3]
	\arrow[from=2-4, to=2-5]
	\arrow[from=2-4, to=3-4]
	\arrow[from=2-5, to=2-6]
	\arrow[from=2-5, to=3-5]
	\arrow["\phi", hook, from=3-1, to=2-3]
	\arrow[from=3-1, to=4-1]
	\arrow[from=3-2, to=3-3]
	\arrow[from=3-3, to=3-4]
	\arrow[from=3-3, to=4-3]
	\arrow[from=3-4, to=3-5]
	\arrow[from=3-4, to=4-4]
	\arrow[from=3-5, to=3-6]
	\arrow[from=3-5, to=4-5]
	\arrow[Rightarrow, no head, from=4-1, to=3-3]
	\arrow[from=4-1, to=5-1]
	\arrow[from=4-2, to=4-3]
	\arrow[from=4-3, to=4-4]
	\arrow[from=4-3, to=5-3]
	\arrow["g", from=4-4, to=4-5]
	\arrow[from=4-4, to=5-4]
	\arrow[from=4-5, to=4-6]
	\arrow[from=4-5, to=5-5]
	\arrow[two heads, from=5-1, to=4-3]
	\arrow[from=5-1, to=6-1]
	\arrow[draw=none, from=5-3, to=5-4]
	\arrow[draw=none, from=5-4, to=5-5]
\end{tikzcd}\]

Here, $Q$ and $Q'$ are projective. Now consider the following commutative diagram

\[\begin{tikzcd}[column sep=10pt]
	& 0 & 0 & 0 \\
	0 & {P_{\Omega_{\Lambda}(L)}} & {P_{\Omega_{\Lambda}(L)}\oplus P_Y} & {P_Y} & 0 \\
	0 & {\Omega_{\Lambda}(L)} & {\Omega_{\Lambda}(M)}\oplus Q & {Y:={\rm Cok} \psi\phi} & 0 \\
	0 & {\Omega_{\Lambda}(K) \oplus Q' } & {\Omega_{\Lambda}(M) \oplus Q} & {\Omega_{\Lambda}(N)} & 0 \\
	& 0 & 0 & 0
	\arrow[from=2-1, to=2-2]
	\arrow[from=2-2, to=2-3]
	\arrow[from=2-3, to=2-4]
	\arrow[from=2-4, to=2-5]
	\arrow[from=3-1, to=3-2]
	\arrow["{\psi\phi}", from=3-2, to=3-3]
	\arrow[from=3-3, to=3-4]
	\arrow[from=4-1, to=4-2]
	\arrow[from=4-2, to=4-3]
	\arrow[from=4-3, to=4-4]
	\arrow[from=4-4, to=4-5]
	\arrow[draw=none, from=5-2, to=5-3]
	\arrow[draw=none, from=5-3, to=5-4]
	\arrow[draw=none, from=1-2, to=1-3]
	\arrow[draw=none, from=1-3, to=1-4]
	\arrow[from=1-2, to=2-2]
	\arrow[from=2-2, to=3-2]
	\arrow[from=1-3, to=2-3]
	\arrow[from=1-4, to=2-4]
	\arrow["{p_Y}", two heads, from=2-4, to=3-4]
	\arrow[from=4-2, to=5-2]
	\arrow[from=4-3, to=5-3]
	\arrow[from=4-4, to=5-4]
	\arrow[from=2-3, to=3-3]
	\arrow[from=3-4, to=3-5]
	\arrow[hook, from=3-2, to=4-2]
	\arrow[Rightarrow, no head, from=3-3, to=4-3]
	\arrow["h", from=3-4, to=4-4]
\end{tikzcd}\]
Then, we have the following commutative diagram:

\begin{equation}\label{equa.thelast.construct.}
\begin{tikzcd}[column sep=10pt]
	0 & {P_{\Omega_{\Lambda}(L)}} && {P_{\Omega_{\Lambda}(L)}\oplus P_Y} && {P_Y} & 0 \\
	{\Omega_{\Lambda}(L)} && {\Omega_{\Lambda}(M)\oplus Q} && {\Omega_{\Lambda}(N)} \\
	0 & {P_L} && {P_L\oplus P_N} && {P_N} & 0 \\
	& L && M && N & 0 \\
	& 0 && 0 && 0
	\arrow[from=4-6, to=4-7]
	\arrow[from=4-6, to=5-6]
	\arrow[from=3-6, to=3-7]
	\arrow["h", from=3-6, to=4-6]
	\arrow["c", from=1-6, to=3-6]
	\arrow["{hp_Y}"', two heads, from=1-6, to=2-5]
	\arrow[hook, from=2-5, to=3-6]
	\arrow[two heads, from=1-4, to=2-3]
	\arrow[two heads, from=1-2, to=2-1]
	\arrow["{ }"', hook, from=2-1, to=3-2]
	\arrow[from=1-1, to=1-2]
	\arrow[from=1-2, to=1-4]
	\arrow[from=1-4, to=1-6]
	\arrow[from=3-1, to=3-2]
	\arrow[from=3-4, to=3-6]
	\arrow[from=3-2, to=3-4]
	\arrow["a", from=1-2, to=3-2]
	\arrow["b", from=1-4, to=3-4]
	\arrow[from=3-2, to=4-2]
	\arrow[from=4-2, to=5-2]
	\arrow[from=1-6, to=1-7]
	\arrow[from=4-4, to=5-4]
	\arrow[from=3-4, to=4-4]
	\arrow[from=4-2, to=4-4]
	\arrow[from=4-4, to=4-6]
	\arrow[hook, from=2-3, to=3-4]
\end{tikzcd}
\end{equation}

 \end{construction}

Hence,  the morphisms $a, b$ and $c$ provide projective presentations for $L, M$, and $N$, respectively.  By  applying the kernel functor to the above commutative diagram,  we obtain  the following left exact sequence

  $$\la': \  \ 0 \rt \Omega^2_{\La}(L) \rt \Omega^2_{\La}(M)\oplus P_1 \rt \Omega^2_{\La}(N)\oplus P_0 $$
  in $\mmod \La,$ where $P_1, P_0 \in {\rm proj}\mbox{-}\La.$  We may obtain different such left exact sequences by our construction. Subsequently, we investigate the uniqueness of one such left exact sequence. To do this, we must first make the necessary preparations as follows.

\vspace{1 mm}

We can regard the right exact sequence $\lambda: L \st{f }\rt M\st{g} \rt N \rt  0$ as an object in ${\rm rep}(A_3, J, \Lambda)$. Furthermore, the last diagram \eqref{equa.thelast.construct.} in the construction leads to the following projective presentation of $\lambda$ in ${\rm rep}(A_3, J, \Lambda)$:
\begin{equation}\label{equ. thelastCon2}
{\small (P_{\Omega_{\La}(L)}\st{1}\rt P_{\Omega_{\La}(L)}\rt 0)\oplus (0 \rt P_Y\st{1}\rt P_Y)\rt (P_L\st{1}\rt P_L\rt 0)\oplus (0 \rt P_N\st{1}\rt P_N)\rt \la \rt 0. }
\end{equation}
Thus, a minimal projective presentation of $\la$ in ${\rm rep}(A_3, J, \Lambda)$ is obtained by removing some projective summands from the two terms of the  projective presentation in  (\ref{equ. thelastCon2}).

 In particular, according to the above observation,  we will get  the following information for the terms of  $\Omega^2\la:(L_2\st{\Omega^2f}\rt M_2\st{\Omega^2g}\rt N_2)$, the second syzygy of $\la$  in ${\rm rep}(A_3, J, \Lambda)$:
 \begin{itemize}
     \item [$(1)$]  $L_2=\Omega^2_{\La}(L)$;
     \item [$(2)$] $M_2=\Omega^2_{\La}(M)\oplus Q_1$ and $N_2=\Omega^2_{\La}(N)\oplus Q_0$, where $Q_1$ and $Q_0$ are   direct summands of $P_0, P_1,$ respectively;
 \end{itemize}
 \noindent
 and moreover, $0 \rt L_2\st{\Omega^2f}\rt M_2\st{\Omega^2g}\rt N_2$ is a left exact sequence.  Due to the uniqueness of a minimal projective presentation, the projective direct summands $Q_1, Q_0$ are uniquely determined by $\la$ up to isomorphism. Therefore,  we can refer to them as $Q_1(\la)$ and $Q_0(\la)$, respectively.

\hspace{1 cm}

\begin{definition}\label{def. typ0-1-P-right-exact}
Let $P$ be a projective $\La$-module. A right short exact sequence $\la:L\rt M\rt N \rt 0$ is called $P$-right exact of type 0 (resp. type 1) if  $Q_0(\la)$ (resp. of $Q_1(\la))$  belongs to $\mathsf{add}P.$ We say that it is a $P$-right exact if it is of both types, i.e, if both $Q_0(\la)$ and  $Q_1(\la)$  belong to $\mathsf{add}P.$
\end{definition}

According to our definition, for any right exact sequence $\la$,  letting $P=Q_0(\la)\oplus Q_1(\la)$, the sequence $\la$ is a $Q$-right exact sequence for any projective module $Q$ such that $Q\in \mathsf{add}P.$ Moreover, the exact sequence in (\ref{equ. thelastCon2}) induces an exact sequence in $\CP(\La)$, as stated in the following lemma.

\begin{lemma}
Let $\la \colon L \rightarrow M \rightarrow N \rightarrow 0$ be a right short exact sequence in $\mmod \La$.  
Then there exists a short exact sequence
\[
\delta \colon 0 \rightarrow \mathsf{L} \rightarrow \mathsf{M} \rightarrow \mathsf{N} \rightarrow 0
\]
in $\CP(\La)$ such that $\mathsf{Cok}(\delta)=\la$, and $\mathsf{M}_c\simeq (Q_1(\la)\rightarrow 0),~\mathsf{N}_c\simeq (Q_0(\la)\rightarrow 0).$
\end{lemma}

It is natural to ask what more we can learn about the second syzygy $\Omega^2(\lambda)$ when $\lambda$ is a short exact sequence. Below, we provide some additional information on this matter.

\begin{proposition}\label{prop. short-exact-seq}
    Let $\la:0 \rt L \rt M \rt N \rt 0$ be a short exact sequence in $\mmod \La.$ Then, the second sygyzy $\Omega^2(\la)$ is also a short exact sequence and $Q_0(\la)=0$. In particular, $\Omega^2(\la)$ has the following form:
    $$0 \rt \Omega^2_{\La}(L)  \rt \Omega^2_{\La}(M)\oplus Q_1(\la) \rt \Omega^2_{\La}(N)\rt 0.$$
\end{proposition}

\begin{proof}
  Let $P_1\st{f}\rt P_0\rt L\rt 0 $ and $Q_1\st{g}\rt Q_0\rt N\rt 0$ be  minimal projective presentations of $L$ and $N$, respectively. Thanks to the horseshoe lemma, we obtain the following commutative diagram
$$\xymatrix{	0 \ar[r] & P_1 \ar[d]^f \ar[r] & P_1\oplus Q_1 \ar[d]^h
\ar[r] &  Q_1\ar[d]^g \ar[r] & 0\\
	0 \ar[r] & P_0 \ar[d] \ar[r] & P_0\oplus Q_0
	\ar[d] \ar[r] & Q_0 \ar[d] \ar[r] & 0\\
0 \ar[r]  & L \ar[d] \ar[r] & M
	\ar[d] \ar[r] & N   \ar[d] \ar[r] & 0\\
& 0  & 0  & 0 & }
$$

We then follow a similar observation discussed in Construction \ref{Construction5} to obtain a projective presentation as given in \eqref{equ. thelastCon2}. We can deduce that the diagram induces a projective presentation for \(\lambda\) as an object in \({\rm rep}(A_3, J, \Lambda)\). By taking the kernel of the obtained projective presentation in \({\rm rep}(A_3, J, \Lambda)\), we obtain the object \(\Delta = (\Omega^2_{\Lambda}(L) \to \Omega^2(M) \oplus P \to \Omega^2_{\Lambda}(N))\) for some projective module \(P\). Due to the property of a minimal projective presentation, which can be considered as a direct summand of any projective presentation, we infer that \(\Omega^2(\lambda)\) is isomorphic to a direct summand of \(\Delta\). This completes the proof.
\end{proof}

Deciding a right exact sequence is a $P$-right exact sequence for which projective module $P$ is somewhat difficult, since we have  to calculate the second syzygy in another category of the abelian ${\rm rep}(A_3, J, \La)$. However, the situation is easier when the algebra is hereditary.

\begin{proposition}\label{Prop. hereditary.P.right}
    Let $\La$ be a hereditary algebra, and let   $\la:L \st{f}\rt M \rt N \rt 0$ be a right  exact sequence in $\mmod \La.$ Set $P=\mathsf{Ker}f$. Then $\la$ is a $P$-right exact sequence.
\end{proposition}
\begin{proof}
Since $\La$ is a hereditary algebra, one can show that the abelian category
${\rm rep}(A_3, J, \La)$ has global dimension~$2$. 
We view $\la$ as an object in this abelian category. 

Since the global dimension of ${\rm rep}(A_3, J, \La)$ is two, 
the second syzygy $\Omega^2 \la$ must be a projective object. 
Using the characterization of projective objects in 
${\rm rep}(A_3, J, \La)$, we infer that $\Omega^2 \la$ is isomorphic to
\[
(P_1\st{1}\rt P_1 \rt 0)\oplus (0 \to P_2 \xrightarrow{\,1\,} P_2)\;\oplus\; (0 \to 0 \to P_3).
\]
However, $(P_1\st{1}\rt P_1 \rt 0)\oplus (0 \to P_2 \xrightarrow{\,1\,} P_2)$ is an injective object in 
${\rm rep}(A_3, J, \La)$. Therefore, $P_1, P_2$ must be zero; otherwise, this would
contradict the fact that $\Omega^2 \la$ is obtained from a minimal projective
resolution of $\la$. So, $Q_1(\la)=0$ and $Q_0(\la)=P_3=\mathsf{Ker} f$. We are done.
\end{proof}

\begin{corollary}\label{Coro. here. Pright.}
Let $\La$ be a hereditary algebra.
Then every short exact sequence
\[
    0 \rt L \rt   M \rt N \rt 0
\]
in $\mmod \La$ is a $P$-right exact sequence for any projective module $P$.

\end{corollary}

\section{ $r$ICE-closed subcategories}\label{Section 6}

In this section, we define the notion of rICE-closed subcategories in the module category $\mmod \Lambda$. Then, we show that an rICE-closed subcategory is indeed the image of an ICE-closed subcategory in $\CP(\Lambda)$. Additionally, we prove that rigid objects in $\CP(\Lambda)$ and $\tau$-rigid modules are in bijection. All of this leads to a version of Enomoto's theorem \cite[Theorem 2.3]{En} for general Artin algebras, providing a bijection between rICE-closed subcategories and basic $\tau$-rigid modules.

 \subsection{ rICE-closed subcategories in $\mmod \La$}
We first need to introduce $P$-morphism with respect to a basic projective module $P$

\begin{definition}\label{Def. P-morphism1245}
Let $P$ be a basic projective module.  A  morphism $f \colon M \to N$ is called a \emph{$P$-morphism} if it  admits a factorization
    \[
    \begin{tikzcd}
        {K } & M && N & {H} & 0 \\
        && D
        \arrow[from=1-1, to=1-2]
        \arrow["f", from=1-2, to=1-4]
        \arrow["g", two heads, from=1-2, to=2-3]
        \arrow[from=1-4, to=1-5]
        \arrow[from=1-5, to=1-6]
        \arrow["h", from=2-3, to=1-4]
    \end{tikzcd}
    \]
    such that the right exact sequences
    \[
        K\rightarrow M \xrightarrow{g} D \rightarrow 0
        \quad \text{and} \quad
        D \xrightarrow{h} N \rightarrow H \rightarrow 0
    \]
    are,  respectively, $P$-right exact and $P$-right exact sequence  of type $1$ (where $h$ need not be a monomorphism, but $g$ must be an epimorphism). Note that $H$ need not be the cokernel of $f$, but it is the cokernel of $h$.
\end{definition}

Now, we can define  the following kinds of subcategories  which play important role in our paper.

\begin{definition}\label{Def. ICE-Mor-Proj}
Let $P$ be a basic projective module and let $\CM$ be a subcategory of $\mmod \La$.
We say that $\CM$ is an \emph{ rICE-closed subcategory} if the following conditions hold.
\begin{itemize}
    \item[$(1)$] 
    For every $P$-right exact sequence of type $0$ in $\mmod \La$
    \[
        M \longrightarrow N \longrightarrow L \longrightarrow 0,
    \]
    if $M, L \in \CM$, then $N \in \CM$.
    In this case, we say that $\CM$ is \emph{closed under $P$-right exact sequences (of type 0)}.

    \item[$(2)$] 
    For every $P$-morphism $f \colon M \to N$ in $\CM$, every module $D$ (resp. $H$) appearing in the factorization as in Definition \ref{Def. P-morphism1245} belongs to $\CM$. In this case, we say that $\CM$ is \emph{closed under images of $P$-morphisms} (resp. {closed under cokernels of $P$-morphisms}).
    
    \item[$(3)$] Every $P$-right exact sequence of type $0$  with terms in $\CM$  is $P$-right exact.

    \item [$(4)$]  The projective module $P$ is \emph{maximal} subject to the properties $(1), (2)~ \text{and} ~(3)$,  that is,
    for every  projective module $Q$ that  $\CM$ satisfying these  three conditions, then   $Q \in \mathsf{add}\, P$.
\end{itemize}
\end{definition}

If $P = \Lambda$ in the above definition, any right exact sequence is automatically a \( P \)-right exact sequence. Hence, in this case, an rICE-closed subcategory  has the following properties. For every right exact sequence \( M \to N \to L \to 0 \) with \( M, L \in \mathscr{M} \), we have \( N \in \mathscr{M} \). Furthermore, for any morphism \( f: M \to N \) in $\CM$ with a factorization \( f: M \xrightarrow{g} D \xrightarrow{h} N \) such that \( g \) is an epimorphism, we have \( D \in \mathscr{M} \). In this case, one can easily see that $\CM$ is an ICE-closed subcategory in $\mmod \La$ (in the usual senses, see Definition \ref{ICE}).

 It may be of interest to determine  rICE-closed subcategories for certain classes of algebras. By Corollary \ref{Coro. here. Pright.}, on a hereditary algebra,  we can easily see that every rICE-closed subcategory is also an ICE-closed subcategory. The converse is  true, as we   will observe in the sequel. Hence, over hereditary algebras, ICE-closed subcategories coincide with rICE-closed subcategories.

\begin{remark}
  We note that if a subcategory $\CM$ is  rICE-closed, then by condition~$(3)$ in the definition, the projective module $P$ is uniquely determined up to isomorphism. In this case, we say $\CM$ is an rICE-closed subcategory with projective module $P.$ Moreover, due to  condition~$(3)$, the second  right exact sequences appearing in  condition~$(2)$ are indeed $P$-right exact.  We also note that the lowercase ``r'' reflects the fact that the notion of rICE-closed subcategories is defined via certain right exact sequences.   
\end{remark}

We also need to revert the notion of ICE-closed subcategories with enough Ext-progenerator in $\CP(\La)$ to $\mmod \La.$ To do this we define the following notation.

\begin{definition}\label{Def. Pext-projective}

Let  $\CM$ be  an rICE-closed subcategory with projective module $P.$  Let $Q$ be a  projective module  such that ${\Hom_{\La}}(Q, M)=0$ for every $M \in \CM.$
\begin{itemize}
\item We call a module $Z \in \CM$, \emph{$r\Ext$-projective in $\CM$} if for every $P$-right exact sequence $X\st{f}\rt Y\st{g}\rt Z\rt 0$ with $Q_0(\la) \in \mathsf{add} \ Q$ in $\CM$,  we  have $f$ is a section   and $g$ a retraction. We denote by $r\mathsf{P}(\CM)$ the subcategory of $\CM$ consisting of all  $r\Ext$-projective modules in $\CM.$

\item  We say $\CM$ has enough $r\Ext$-projective if for every module  $X \in \CM$, there is a $P$-right  exact sequence
$$\la: N\rt Y\rt X\rt 0, $$
such that $Q_1(\la) \in \mathsf{add}Q,\  N\in \CM$ and $Y \in r\mathsf{P}(\CM)$.
    \item We say that $\CM$ has a \emph{r$\Ext$-progenerator} if $\CM$ has enough  $r\Ext$-projectives, and there exists  $M \in \CM$ such that $\mathsf{add}\ M=r\mathsf{P}(\CM)$. Moreover, let $M_0$ be a maximal projective direct summand of $M$. There is a $M_0^*$-right exact sequence in $\mmod \La^{\rm op}$
\begin{equation}\label{eq. def-pexts}
       P^*\rt \mathsf{Tr}M_1\oplus Q^*_1 \rt \mathsf{Tr}M_2 \oplus Q^*_2,
  \end{equation}  
    where $Q_1, Q_2 \in \mathsf{add}\ Q$ and $M_1, M_2 \in \mathsf{add}\ M.$     In this case,  we say $M$ is a \emph{$r\Ext$-progenerator} (along with $Q$).
\end{itemize}
\end{definition}

We first need the next two lemmas.  

\begin{lemma}\label{lem. istau-tilting objects}
    Let $M$ be a module and $P$ a projective module. Then, ${\Hom_{\La}}(P, M)=0$ if and only if ${\Hom_{\La}}(M, \nu P)=0.$
\end{lemma}
\begin{proof}
  $(\Rightarrow)$ Assume that there exists a non-zero map $g: M \rightarrow \nu P$. Then we have a non-section morphism $(0, 1): (0 \rightarrow \nu P) \rightarrow (\text{Im} g \hookrightarrow \nu P)$ in the morphism category $\mathscr{H}$. Due to the almost split morphism given in \cite[Proposition 5.2]{HW}, we have the following:
\begin{equation}\label{diag. exa. lem}
\xymatrix@1{   (0 \rightarrow \nu P)
			\overset{(0,~1)}\rightarrowtail
			  (P\stackrel{f}\rightarrow \nu P)\overset{(1,~0)}\twoheadrightarrow
			(P\rightarrow 0). }
\end{equation}
The non-section morphism $(0, 1)$ has to be factored through the monomorphism in the almost split sequence \eqref{diag. exa. lem}. This gives us a non-zero map from $P$ to $\text{Im} g$, which in turn lifts to a non-zero map from $P$ to $M$. However, this contradicts the hypothesis that ${\Hom_{\Lambda}}(P, M) = 0$.

$(\Leftarrow)$ Assume that there exists a non-zero map $h: P \rightarrow M$. It induces a non-retraction $(1, 0): (P \stackrel{h}{\rightarrow} M) \rightarrow (P \rightarrow 0)$. Hence, it factors through the epimorphism in the sequence \eqref{diag. exa. lem} via a morphism, denoted as $(1, d): (P \stackrel{h}{\rightarrow} M) \rightarrow (P \stackrel{f}{\rightarrow} \nu P)$. This gives $dh = f$. However, according to our hypothesis, $d = 0$. It follows that $f = 0$, which is a contradiction.
\end{proof}

\medskip

\begin{lemma}\label{lem. vanishing-Hom-Ext}
Let $P$ be a projective module and $\mathsf{M}$ an object in $\CP(\La)$. Then, ${\Hom_{\La}}(P, \mathsf{Cok}\mathsf{M})=0$ if and only if ${\Ext^1_{\CP}}((P\rt 0), \mathsf{M})=0$.
\end{lemma}

\begin{proof}
 Let $\mathsf{P}=(P\rt 0)$. By  \cite[Proposition 5.2]{HW},  $\tau_{\mathscr{P}} \mathsf{P}=(P_1\st{d}\rt P_0)$, with $\mathsf{Cok} d=\nu P.$ We have 
\begin{align*}
     \Ext^1_{\CP}( \mathsf{P}, \mathsf{M}) & \simeq D\overline{{\rm Hom}}_{\mathscr{P}}(\mathsf{M}, \tau_{\mathscr{P}} \mathsf{P}) \\ 
     & \simeq D\Hom_{\Lambda}(\mathsf{Cok}\mathsf{M}, \mathsf{Cok}\tau_{\mathscr{P}}\mathsf{P})\\ & = D{\Hom_{\Lambda}}(\mathsf{Cok}\mathsf{M}, \nu P). 
\end{align*}
         
Hence, $\Ext^1_{\CP}(\mathsf{P}, \mathsf{M})=0$ if and only if ${\Hom_{\La}}(\mathsf{Cok}\mathsf{M}, \nu P)=0$. On the other hand, by Lemma \ref{lem. istau-tilting objects}, we know that ${\Hom_{\La}}(\mathsf{Cok}\mathsf{M}, \nu P)=0$ if and only if ${\Hom_{\La}}(P, \mathsf{Cok}\mathsf{M})=0$. Thus, we obtain the desired result.
\end{proof}

Let \( M \) be a module with a minimal projective presentation \( P_1 \xrightarrow{f} P_0 \to M \to 0 \). We denote by \( \mathsf{X}_M \) the object \( (P_1 \xrightarrow{f} P_0) \) in \( \CP(\Lambda) \). Let \( \CM \) be an rICE-closed subcategory with projective module $P$.  We associate the following subcategory in \( \CP(\Lambda) \) to \( \CM \):

$$\mathsf{ICE}(\CM)=\mathsf{add}\{ \mathsf{X}_M\mid M \in \CM\}\cup \{(P\rt 0), (\La\st{1}\rt \La)\}.$$

\begin{proposition}\label{prop. ICE-Mor-relative-ICE}
Let  $\CM$ be  an rICE-closed subcategory  with  projective module $P$ (resp.  has a $r\Ext$-progenerator) in  \( \mmod \Lambda \). Then $\mathsf{ICE}(\CM)$ is an ICE-closed subcategory (resp. has a $\Ext$-progenerator) in $\CP(\La).$
\end{proposition}
\begin{proof} Assume that \(\CM\) is an rICE-closed subcategory with  projective module \(P\). We show that \(\mathsf{ICE}(\CM)\) is closed under extensions. Consider a conflation

\[
\lambda: \mathsf{X} \rightarrowtail \mathsf{Y} \twoheadrightarrow \mathsf{Z}
\]
\noindent
in $\CP$ with \(\mathsf{X}, \mathsf{Z} \in \mathsf{ICE}(\CM)\). By taking the cokernel, we obtain the right exact sequence

\[
\mathsf{Cok}\lambda: \mathsf{Cok}\mathsf{X} \to \mathsf{Cok}\mathsf{Y} \to \mathsf{Cok}\mathsf{Z} \to 0
\]
\noindent
in \(\mmod \Lambda\). Since $\mathsf{Z}_c \in \mathsf{add}(P\rt 0),$ we infer   that the right exact sequence  $\mathsf{Cok}\lambda$ is a $P$-right exact sequence of type $0$.  This implies that \(\mathsf{Cok}\mathsf{Y}\) lies in \(\CM\), due to the first condition of an rICE-closed subcategory. Hence, \(\mathsf{Y}_0 \oplus \mathsf{Y}_a \in \mathsf{ICE}(\CM)\). Also, $\mathsf{Y}_b \in \mathsf{ICE}(\CM)$  since $(\La \st{1} \rt  \La)$ is in $\mathsf{ICE}(\CM)$. It remains to show that $\mathsf{Y}_c \in \mathsf{ICE}(\CM).$

The sequence \(\lambda\) induces a projective presentation of the right exact sequence \(\mathsf{Cok}\lambda\) as an object in \({\rm rep}(A_3, J, \Lambda)\). Indeed, assume that $\mathsf{X}=(P_1\st{f}\rt P_0), \mathsf{Z}=(Q_1\st{g}\rt Q_0)$ and $\mathsf{Y}=(P_1\oplus Q_1 \st{h}\rt P_0\oplus Q_0)$. The  sequence $\la$, together with  the snake lemma,  yields  the following communicative diagram:

$$\xymatrix{0 \ar[r] & \mathsf{Ker} f \ar[d] \ar[r] & \mathsf{Ker} h \ar[d]
\ar[r] &  \mathsf{Ker} g\ar[d] & \\	0 \ar[r] & P_1 \ar[d]^f \ar[r] & P_1\oplus Q_1 \ar[d]^h
\ar[r] &  Q_1\ar[d]^g \ar[r] & 0\\
	0 \ar[r] & P_0 \ar[d] \ar[r] & P_0\oplus Q_0
	\ar[d] \ar[r] & Q_0 \ar[d] \ar[r] & 0\\
  & \mathsf{Cok}\ f \ar[d] \ar[r] & \mathsf{Cok}\ h
	\ar[d] \ar[r] & \mathsf{Cok}\ g   \ar[d] \ar[r] & 0\\
& 0  & 0  & 0 & }
$$
The second and third rows (from the top)  provide a projective presentation of $\mathsf{Cok}\la$ as an object in \({\rm rep}(A_3, J, \Lambda)\).
Using  the properties of  minimal projective presentations and the classification of projective objects in \({\rm rep}(A_3, J, \Lambda)\), we  deduce  the following isomorphism in  \({\rm rep}(A_3, J, \Lambda)\):

$$(\mathsf{Ker}\ f \rt \mathsf{Ker}\ h \rt \mathsf{Ker}\ g)\simeq \Omega^2(\la) \oplus (Q\st{1}\rt Q\rt 0)\oplus (0\rt Q'\st{1}\rt Q')\oplus (0\rt 0\rt Q'').$$
\noindent
Since $\mathsf{X}$ and $\mathsf{Z}$ belong to  \(\mathsf{ICE}(\CM)\), it follows that  $\mathsf{X}_c$ and $\mathsf{Z}_c$ are in $\mathsf{add} (P\rt 0).$ Hence, $Q\oplus Q'\oplus Q'' \in \mathsf{add}P$. By the third condition of an rICE-closed subcategory, the exact right sequence $\mathsf{Cok}\  \la$ is also of type $1$. Therefore, $Q_1(\mathsf{Cok}\ \la) \in \mathsf{add}P$.  From the above isomorphism, we  obtain $\mathsf{Y}_c \simeq (Q_1(\mathsf{Cok}\ \la)\oplus Q\oplus Q'\rt 0)$, as desired. This shows  that $\mathsf{Y}$ is an object in  \(\mathsf{ICE}(\CM)\).

Now we show that $\mathsf{ICE}(\CM)$ is admissible cokernel-closed and admissible image-closed.  
Assume that $f \colon \mathsf{M} \to \mathsf{N}$ is an admissible morphism in $\mathsf{ICE}(\CM)$. 
Thus, we have a factorization
\[
f \colon \mathsf{M} \xrightarrow{g} \mathsf{D} \xrightarrow{h} \mathsf{N}
\]
in $\CP$, where $g$ is a deflation and $h$ is an inflation.  
According to this factorization, we obtain the following conflations:
\[
\delta_1: \mathsf{K} \rightarrowtail \mathsf{M} \st{g}\twoheadrightarrow \mathsf{D}
\]
\[
\delta_2: \mathsf{D} \st{h}\rightarrowtail \mathsf{N} \twoheadrightarrow \mathsf{H}
\]
\noindent
Taking cokernels of the conflations $\delta_1$ and $\delta_2$, we obtain the induced factorization
\[
\begin{tikzcd}
    \mathsf{Cok}\,\mathsf{K} & \mathsf{Cok}\,\mathsf{M} && \mathsf{Cok}\,\mathsf{N} & \mathsf{Cok}\,\mathsf{H} & 0 \\
    && \mathsf{Cok}\,\mathsf{D}
    \arrow[from=1-1, to=1-2]
    \arrow["\mathsf{Cok}\,f", from=1-2, to=1-4]
    \arrow["\mathsf{Cok}\,g", two heads, from=1-2, to=2-3]
    \arrow[from=1-4, to=1-5]
    \arrow[from=1-5, to=1-6]
    \arrow["\mathsf{Cok}\,h", from=2-3, to=1-4]
\end{tikzcd}
\]

\noindent
Since $\mathsf{M}_c$ belongs to $\mathsf{add}(P \to 0)$, it follows that $\mathsf{Cok}\,\delta_1$ is a $P$-morphism of type~$1$.  
On the other hand, since $\mathsf{D}_c$ is an injective object, the conflation $\delta_2$ implies that $\mathsf{D}$ is a direct summand of $\mathsf{N}$.  
Because $\mathsf{N} \in \mathsf{ICE}(\CM)$, we obtain that $\mathsf{N}_c \in \mathsf{add}(P \to 0)$, and hence also
\[
\mathsf{D}_c \in \mathsf{add}(P \to 0).
\]
This shows that $\mathsf{Cok}\,\delta_1$ is also of type~$0$. Therefore, $\mathsf{Cok}\,f$ is a $P$-morphism.

This allows us to apply the second condition in Definition~\ref{Def. ICE-Mor-Proj}.  
By the $P$-morphism $\mathsf{Cok}\,f$, it follows that $\mathsf{Cok}\,\mathsf{D}$ lies in $\CM$.  
This implies that $\mathsf{D}_0 \oplus \mathsf{D}_a \in \mathsf{ICE}(\CM)$.  
The direct summand $\mathsf{D}_b$ automatically lies in $\mathsf{ICE}(\CM)$.  
Moreover, as discussed earlier, $\mathsf{D}_c \in \mathsf{ICE}(\CM)$.  
Hence, $\mathsf{D}$ is an object of $\mathsf{ICE}(\CM)$. Consequently, $\mathsf{ICE}(\CM)$ is closed under admissible images.

Now we show that $\mathsf{H}$ also lies in $\mathsf{ICE}(\CM)$.  
Since $\mathsf{N}_c$ belongs to $\mathsf{add}(P \to 0)$, it follows that the right exact sequence 
$\mathsf{Cok}\,\delta_2$ is of type~$1$.  
Hence, the second condition in Definition~\ref{Def. ICE-Mor-Proj} guarantees that 
$\mathsf{Cok}\,\mathsf{H}$ lies in $\CM$.  
This implies that $\mathsf{H}_0 \oplus \mathsf{H}_a \in \mathsf{ICE}(\CM)$.  

For the direct summand $\mathsf{H}_b$, there is no issue.  
Moreover, since all terms of the right exact sequence $\mathsf{Cok}\,\delta_2$ belong to $\CM$, 
we obtain that $\mathsf{Cok}\,\delta_2$ is also of type~$0$.  
It follows that $\mathsf{H}_c$ lies in $\mathsf{ICE}(\CM)$.  Hence, $\mathsf{H} \in \mathsf{ICE}(\CM)$, as desired.

Finally, let $\CM$ be an rICE-closed subcategory with projective module $P$ that admits an $r\Ext$-progenerator $M$ along with projective module $Q$ (as in Definition~\ref{Def. Pext-projective}). We show that $\mathsf{ICE}(\CM)$ is an ICE-closed subcategory which has an $\Ext$-progenerator in $\CP(\La)$. It suffices to prove the existence of an $\Ext$-progenerator.

Set
\[
\mathsf{M}=\mathsf{X}_M\oplus (Q\rt 0)\oplus (\La\st{1}\rt \La).
\]
We prove that $\mathsf{M}$ acts as an $\Ext$-progenerator for $\mathsf{ICE}(\CM)$. First, we show that $\mathsf{M}$ is $r\Ext$-projective in $\CM$. The case of the direct summand $(\La\st{1}\rt \La)$ is clear. Since ${\Hom_{\La}}(Q,M)=0$, Lemma~\ref{lem. vanishing-Hom-Ext} implies that
\[
{\Ext^1_{\CP}}((Q\rt 0),\mathsf{N})=0
\]
for every $\mathsf{N}\in \mathsf{ICE}(\CM)$. Hence the claim also holds for the direct summand $(Q\rt 0)$.

For the direct summand $\mathsf{X}_M$, let
\[
\la:\quad 0 \rt \mathsf{Y}\st{f}\rt \mathsf{Z}\st{g}\rt \mathsf{X}_M\rt 0
\]
be an extension in $\CP$. The term $\mathsf{X}_M$ clearly has no injective direct summand, and without loss of generality we may assume that $\mathsf{Y}$ has no injective direct summand. In this case, the sequence $\la$ induces a projective presentation in ${\rm rep}(A_3,J,\Lambda)$ of the right exact sequence
\[
\mathsf{Cok}\,\la:\quad 
\mathsf{Cok}\,\mathsf{Y}\st{\mathsf{Cok}f}\rt 
\mathsf{Cok}\,\mathsf{Z}\st{\mathsf{Cok}g}\rt 
\mathsf{Cok}\,\mathsf{X}_M\rt 0.
\]
Note that $\mathsf{Cok}\,\mathsf{X}_M=M$. By assumption, $\mathsf{Cok}f$ and $\mathsf{Cok}g$ are a section and a retraction, respectively. Hence we obtain an isomorphism in ${\rm rep}(A_3,J,\Lambda)$:
\[
\mathsf{Cok}\,\la 
\simeq 
(A\st{1}\rt A\rt 0)\oplus (0\rt B\st{1}\rt B)
\]
for some modules $A$ and $B$. Computing the minimal projective presentation of the right-hand side in ${\rm rep}(A_3,J,\Lambda)$, we see that it has the form
\[
(P_1\st{1}\rt P_1\rt 0)\oplus (0\rt Q_1\st{1}\rt Q_1)
\lrt 
(P_0\st{1}\rt P_0\rt 0)\oplus (0\rt Q_0\st{1}\rt Q_0)
\rt 
\mathsf{Cok}\,\la \rt 0.
\]
By the uniqueness of minimal projective presentations up to isomorphism, the minimal projective presentation induced by $\la$ is isomorphic to the above one. It follows that the sequence $\la$ splits, as desired.

Let $\mathsf{Y}$ be an object in $\mathsf{ICE}(\CM)$. 
We need to show that there exists a short exact sequence in $\CP$
\begin{equation}\label{eq. short55}
    0 \to \mathsf{K} \to \mathsf{N} \to \mathsf{Y} \to 0 
\end{equation}    
such that $\mathsf{N} \in \mathsf{add} \,\mathsf{M}$. 

The case when $\mathsf{Y} = \mathsf{Y}_b$ is clear. 
Assume that $\mathsf{Y} = \mathsf{Y}_0 \oplus \mathsf{Y}_a$. 
By assumption, in $\mmod \La$, there is a $Q$-right exact sequence
\[
K \to N \to \mathsf{Cok}\, \mathsf{Y} \to 0
\]
with $Q_1(\lambda) \in \mathsf{add}\, Q$, $K \in \CM$, and $N \in \mathsf{add}\, M$. 
In view of Construction~\ref{Construction5}, there exists a short exact sequence as in \eqref{eq. short55} such that 
$\mathsf{N}_c \in \mathsf{add}\,(Q \to 0)$, 
$\mathsf{N}_0 \oplus \mathsf{N}_a \in \mathsf{add}\, \mathsf{X}_M$, 
and $\mathsf{K} =\mathsf{X}_K$. 
Thus, $\mathsf{N} \in \mathsf{add}\, \mathsf{M}$ and $\mathsf{K} \in \mathsf{ICE}(\CM)$, 
and we obtain the desired short exact sequence in $\mathsf{ICE}(\CM)$.

If $\mathsf{Y} = \mathsf{Y}_c = (P \to 0)$, then by assumption and 
in view of Construction~\ref{Construction5}, we obtain the following short exact sequence in $\CP(\La^{\rm op})$:
\begin{equation}\label{eq. shorexact4}    
0 \to (0 \to P^*) \to (P_0^* \xrightarrow{f^*} P_1^*) \oplus (0 \to Q_1^*) 
\to ((P_0')^* \xrightarrow{g^*} (P_1')^*) \oplus (0 \to Q_2^*) \to 0,
\end{equation}
where $Q_1, Q_2 \in \mathsf{add}\, Q$, and 
$P_1 \xrightarrow{f} P_0 \to M_1 \to 0$ and 
$P_1' \xrightarrow{g} P_0' \to M_2 \to 0$ are minimal projective presentations. 

By applying the duality functor $(-)^* : \CP(\La^{\rm op}) \to \CP(\La)$ on (\ref{eq. shorexact4}), 
we obtain the desired short exact sequence in \eqref{eq. short55}. 
For the case when $\mathsf{Y}$ is an object in $\mathsf{add}\,(P \to 0)$, 
the same short exact sequence as in \eqref{eq. shorexact4} exists, 
and the same argument applies. Therefore, we conclude that $\mathsf{M}$ is an $\Ext$-progenerator in $\mathsf{ICE}(\CM)$.
\end{proof}
\begin{remark}
 According to the construction of $\mathsf{ICE}(\CM)$, it is the maximal ICE-closed subcategory of $\CP(\Lambda)$ whose image under the cokernel functor is equal to $\CM$.
\end{remark}
Let $\CC$ be an ICE-closed subcategory of $\CP(\La)$. We denote by $\mathsf{icep}(\CC)$ the basic projective module $Q$ such that the injective object $(Q \to 0)$ belongs to every ICE-closed subcategory $\CC'$  such that  $\CC \subseteq \CC' \subseteq \CP(\La)$ and    $ \mathsf{Cok}\,\CC = \mathsf{Cok}\,\CC'.$

\begin{proposition}\label{prop. ICE-Mor-relative-ICE2}
 
Let $\CC$ be an ICE-closed subcategory (resp.  has a  $\Ext$-progenerator) in $\CP(\La)$. Then, $\mathsf{Cok} \CC$ is an rICE-closed subcategory with the  projective module \(\mathsf{icep}(\mathscr{C})\) (resp.  has a    $r\Ext$-progenerator) in $\mmod \La$.
\end{proposition}

\begin{proof}
     Assume that \(\CC\) is an ICE-closed subcategory of \(\CP\). 
Without loss of generality, we may assume that \(\CC\) contains the projective object 
\((\Lambda \xrightarrow{1} \Lambda)\). 
Set \(P := \mathsf{icep}(\CC)\). 
We show that \(\mathsf{Cok}\,\CC\) is an rICE-closed subcategory with projective module \(P\).

Let 
\begin{equation}\label{eq. shorttype0}
 \lambda:\ M \rightarrow N \rightarrow L \rightarrow 0
\end{equation}
be a \(P\)-right exact sequence of type \(0\) in $\mmod \La$ 
with \(M, L \in \mathsf{Cok}\CC\). 
According to Construction~\ref{Construction5}, we obtain a conflation  
\[
\lambda':\  \mathsf{M} \rightarrowtail \mathsf{N} \twoheadrightarrow \mathsf{L}
\]
in \(\CP\) such that \(\mathsf{Cok}\,\lambda' = \lambda\), \(\mathsf{M} \in \CC\), and 
\(\mathsf{L}_0 \oplus \mathsf{L}_a \oplus \mathsf{L}_b \in \CC\). 
Since \(\lambda\) is of type \(0\), we have  $\mathsf{L}_c \in \mathsf{add}(P \to 0).$
By the definition of the projective module \(P\), there exists an ICE-closed subcategory 
\(\CC'\) of \(\CP\) containing \(\CC\) and such that \((P \to 0) \in \CC'\). 
Hence, we may regard the short exact sequence \(\lambda'\) as lying in \(\CC'\). 
Since \(\CC'\) is closed under extensions, it follows that $\mathsf{N} \in \CC'.$ Consequently, $
\mathsf{Cok}\,\mathsf{N} \in \mathsf{Cok}\,\CC' 
= \mathsf{Cok}\,\CC.$ This proves that \(\mathsf{Cok}\,\CC\) is closed under \(P\)-right exact sequences (of type 0).

Let $f \colon M \to N$ be a $P$-morphism. Then it admits a factorization as in Definition~\ref{Def. P-morphism1245}:
 \begin{equation*}
         \begin{tikzcd}
        {K } & M && N & {H} & 0 \\
        && D
        \arrow[from=1-1, to=1-2]
        \arrow["f", from=1-2, to=1-4]
        \arrow["g", two heads, from=1-2, to=2-3]
        \arrow[from=1-4, to=1-5]
        \arrow[from=1-5, to=1-6]
        \arrow["h", from=2-3, to=1-4]
    \end{tikzcd}
    \end{equation*}
such that the right exact sequences
\[
\delta_1:\ K \rightarrow M \xrightarrow{g} D \rightarrow 0
\quad \text{and} \quad
\delta_2:\ D \xrightarrow{h} N \rightarrow H \rightarrow 0
\]
are $P$-right exact, and $\delta_2$ is of type~$1$.

By Construction~\ref{Construction5}, there exist conflations 
\[
\delta'_1:\   \mathsf{K} \rightarrowtail \mathsf{M} 
\st{g'}\twoheadrightarrow \mathsf{D}' 
\]

and
\[
\delta'_2:\  \mathsf{D} \st{h'}\rightarrowtail 
\mathsf{N} \twoheadrightarrow \mathsf{H} 
\]
in $\CP$ such that  $ \mathsf{Cok}\,\delta'_1 = \delta_1,   \mathsf{Cok}\,\delta'_2 = \delta_2,$
and 
\[
\mathsf{D}' = \mathsf{D} \oplus \bigl(Q_0(\delta_1) \to 0\bigr).
\]
By adding the split exact sequence
\[
0 \longrightarrow (Q_0(\delta_1) \to 0)
\xrightarrow{\mathrm{id}}
(Q_0(\delta_1) \to 0)
\longrightarrow 0 \longrightarrow 0
\]
to $\delta'_2$, we obtain the following short exact sequence in $\CP$:
\[
\delta''_2:\ 
0 \longrightarrow 
\mathsf{D} \oplus (Q_0(\delta_1) \to 0)
\xrightarrow{h''}
\mathsf{N} \oplus (Q_0(\delta_1) \to 0)
\longrightarrow 
\mathsf{H}
\longrightarrow 0.
\]
\noindent
Combining the sequences $\delta'_1$ and $\delta''_2$, we obtain the following factorization in $\CP$:
\begin{equation*}
   \begin{tikzcd}
\mathsf{M} \arrow[r, "f'"] \arrow[d, two heads, "g'"'] 
& \mathsf{N} \oplus (Q_0(\delta_1) \to 0) \\
\mathsf{D}' \arrow[ur, tail, "h''"'] &
\end{tikzcd}
\end{equation*}
where $f' = h'' g'$. Hence $f'$ is an admissible morphism in $\CC$. Since $\CC$ is closed under admissible images and cokernels, it follows that $\mathsf{D}'$ and $\mathsf{H}$ both belong to $\CC$. Consequently,
\[
D = \mathsf{Cok}\,\mathsf{D}' 
\quad \text{and} \quad 
H = \mathsf{Cok}\,\mathsf{H}
\]
belong to $\mathsf{Cok}\,\CC$, as desired.

Now we verify condition~$(3)$ in Definition~\ref{Def. ICE-Mor-Proj}. 
Let \(\lambda\) be a \(P\)-right exact sequence of type 0 as in~(\ref{eq. shorttype0}). 
By a similar argument as above, there exists a conflation
\[
\lambda':\  \mathsf{M} \rightarrowtail \mathsf{N} \twoheadrightarrow \mathsf{L}
\]
in \(\CP\) such that \(\mathsf{Cok}\,\lambda' = \lambda\) and $\mathsf{N}_c$ belongs to \(\mathsf{add}(P \to 0)\). 
This implies that \(\lambda\) is also of type~\(1\), hence a \(P\)-right exact sequence. 

Finally, by the definition of the projective module $P$, it follows that 
$P$ is the maximal projective module required in 
Definition~\ref{Def. ICE-Mor-Proj}. 
This completes the proof that  \(\mathsf{Cok}\,\CC\) is an rICE-closed subcategory with projective module \(P\).

Now assume that $\CC$ admits an $\Ext$-progenerator $\mathsf{M}$. 
Let $\mathsf{M}_c = (Q \to 0)$. We show that $M := \mathsf{Cok}\,\mathsf{M}$
is an $r\Ext$-progenerator, along with $Q$, for $\mathsf{Cok}\,\CC$. Let  $
\lambda:~X \xrightarrow{f} Y \xrightarrow{g} Z \to 0$ be a $P$-right exact sequence in $\CM$ with $Q_0(\lambda) \in \mathsf{add}\, Q$.  
By Construction~\ref{Construction5}, there exists a short exact sequence
\[
\lambda':\quad 
0 \to \mathsf{X} \xrightarrow{f'} \mathsf{Y} \xrightarrow{g'} \mathsf{Z} \to 0
\]
in $\CP$ such that $\mathsf{Cok}\,\lambda' = \lambda$. Since $Q_0(\lambda) \in \mathsf{add}\, Q$ and $\lambda$ is $P$-right exact, we may assume that $\lambda'$ lies in $\CC$ and that 
$\mathsf{Z}_c \in \mathsf{add}\,(Q \to 0)$.  As $\mathsf{Z}$ is $\Ext$-projective in $\CM$, and therefore  $\lambda'$ splits in $\CP$.  Consequently, $f$ is a section and $g$ is a retraction, as desired. 
Let $N \in \mathsf{Cok}\,\CC$. Then $\mathsf{X}_N$ lies in $\CC$. 
Since $\mathsf{M}$ is an $\Ext$-progenerator for $\CC$, there exists a short exact sequence
\[
\delta:\quad 0 \to \mathsf{K} \to \mathsf{S} \to \mathsf{X}_N \to 0
\]
with $\mathsf{K}, \mathsf{S} \in \mathsf{add}\,\mathsf{M}$. Because $\mathsf{S}_c \in \mathsf{add}(Q \to 0)$, it follows that 
$\mathsf{Cok}\,\delta$ satisfies the required condition in 
Definition~\ref{Def. Pext-projective}. Finally, since $(P \to 0)$ is an object of $\CC$ and $\CC$ has enough $\Ext$-projective objects, there exists a short exact sequence in $\CC$ ending at $(P \to 0)$. Applying the duality functor $(-)^* : \CP(\Lambda) \to \CP(\Lambda^*),$ 
we obtain the required exact sequence~\eqref{eq. def-pexts} 
from Definition~\ref{Def. Pext-projective}.
\end{proof}

\subsection{The bijection: $\tau$-rigid modules and rICE-closed subcategories}

\begin{lemma}\label{rigid-tau-rigid}
The following statements hold.
\begin{itemize}
   \item [$(i)$] Let $\mathsf{M}$ be  a rigid object in $\CP(\La)$. Then $\mathsf{ Cok}\mathsf{M}$ is a $\tau$-rigid module;
   \item [$(ii)$] Let  $M$ be  a module in $\mmod \La$ with a minimal projective presentation $P\st{f}\rt Q\rt M \rt 0$. If $M$ is a $\tau$-rigid module, then  $\mathsf{X}_M=(P\st{f}\rt Q)$ is rigid in $\CP(\La).$
    \end{itemize}
\end{lemma}
  \begin{proof}
  $(i)$  If $\mathsf{M}_0=0$,  there is nothing to prove. So, we assume that $\mathsf{M}_0\neq 0$. In view of Proposition \ref{Cok-ass}, we have
    \begin{align*}
   \tau\mathsf{Cok \mathsf{M}}&= \tau\mathsf{Cok}(\mathsf{M}_0\oplus \mathsf{M}_a\oplus \mathsf{M}_b\oplus \mathsf{M}_c)&\\
   &=\tau\mathsf{Cok}\mathsf{M}_0\oplus \tau \mathsf{Cok}\mathsf{M}_a\oplus \tau \mathsf{Cok}\mathsf{M}_b\oplus \tau\mathsf{Cok}\mathsf{M}_c\\
   &= \mathsf{Cok}\tau_{\CP}\mathsf{M}_0.
    \end{align*}
   Note that  $\mathsf{Cok}\mathsf{M}_a$ is projective, so $\tau \mathsf{Cok}\mathsf{M}_a=0$.  Let $f: \mathsf{Cok}\mathsf{M} \rt \tau \mathsf{Cok}\mathsf{M}$ be a morphism in $\mmod \La$, which is indeed a morphism from $\mathsf{Cok}\mathsf{M}$ to $\mathsf{Cok} \tau_{\CP}\mathsf{M}_0$ by the above observation.  We can lift $f$ to the projective presentations of $\mathsf{Cok}\mathsf{M}$ and $\mathsf{Cok}\tau_{\CP}\mathsf{M}_0$ induced by the objects $\mathsf{M}$ and $\tau_{\CP}\mathsf{M}_0$. Denote by $g:\mathsf{M}\rt \tau_{\CP}\mathsf{M}_0$ the lifted map. Hence $\mathsf{Cok}g=f$. Since $\mathsf{M}$ is rigid, by the  ARS-duality in $\CP$ (see Subsection \ref{Sub. remainder}), we have  $\overline{{\rm Hom}}_{\CP}(\mathsf{M}, \tau_{\CP} \mathsf{M}_0)=0.$ This implies  that $g$ factors through an injective object in $\CP$. Since the cokernel of injective objects is zero, we infer that $\mathsf{Cok}g=0$. Consequently, $f=0$, as desired.
   
   $(ii)$ Let $M$ be a $\tau$-rigid module. According to ARS-duality in $\CP$, we have:
$$\text{Ext}^1_{\mathscr{P}}(\mathsf{X}_M, \mathsf{X}_M) \simeq D\underline{\text{Hom}}_{\mathscr{P}}(\tau^{-1}_{\mathscr{P}}\mathsf{X}_M, \mathsf{X}_M) \simeq D\overline{\text{Hom}}_{\mathscr{P}}(\mathsf{X}_M, \tau_{\mathscr{P}}\mathsf{X}_M).$$
By Theorem \ref{Thm-Cok-equiv}, $\overline{\text{Hom}}_{\mathscr{P}}(\mathsf{X}_M, \tau_{\mathscr{P}}\mathsf{X}_M) \simeq \text{Hom}_{\Lambda}(\mathsf{Cok}\mathsf{X}_M, \mathsf{Cok}\tau_{\mathscr{P}}\mathsf{X}_M)$. Using Proposition \ref{Cok-ass}, we know that $\mathsf{Cok}\tau_{\mathscr{P}}\mathsf{X}_M\simeq \tau M$. Since $M$ is $\tau$-rigid, we can conclude that $\text{Ext}^1_{\mathscr{P}}(\mathsf{X}_M, \mathsf{X}_M) \simeq D \text{Hom}_{\Lambda}(M, \tau M) = 0$, as desired.
  \end{proof}

In the next result, the first $\mathsf{Cok}\ \overline{M}$  means that the subcategory of $\CP(\La)$ consisting of all cokernels of morphisms in $\mathsf{add} \ \overline{M}$. But, the second $\mathsf{Cok}$ means that applying the cokernel functor on the subcategory $\mathsf{Cok} \ \overline{M}$.

\begin{theorem}\label{mainTheorem5}
Let $\La$ be an Artin algebra.
Then there exist mutually inverse bijections between the following two sets:
\begin{itemize}
    \item[(1)] Isomorphism classes of $\tau$-rigid modules;
    \item[(2)]  rICE-closed subcategories of $\mmod \La$ that admit an
    $r\Ext$-progenerator.
\end{itemize} 
The map from $(1)$ to $(2)$ is given by $M \mapsto \mathsf{Cok}\bigl(\mathsf{Cok}\,\overline{M}\bigr),$ where $ \overline{M}=\mathsf{X}_M\oplus(Q\rt 0)\oplus(\La\st{1}\rt \La)$ and $Q$ is the maximal projective module such that ${\Hom_{\La}}(Q, M)=0$.  Conversely, the map from $(2)$ to $(1)$ is given by $ \CM \mapsto  \mathsf{Cok}\bigl(\mathsf{P}(\mathsf{ICE}(\CM))\bigr)$.
\end{theorem}
\begin{proof}
First, we note that by Lemmas~\ref{lem. vanishing-Hom-Ext} 
and~\ref{rigid-tau-rigid}, the object $\overline{M}$ is rigid in $\CP$. 
Using Propositions~\ref{prop. ICE-Mor-relative-ICE} 
and~\ref{prop. ICE-Mor-relative-ICE2} together with 
Theorem~\ref{Thm. ICE-rigid-Mor}, we see that the maps defined in the statement are well defined. 
It is straightforward to verify that these maps are mutually inverse.
\end{proof}

In the rest, we try to provide some more information for the subcategory $\mathsf{Cok}\bigl(\mathsf{Cok}\,\overline{M}\bigr)$.

\begin{proposition}\label{Propp. inc}
Let $M$ be a $\tau$-rigid module, and  $ \overline{M}=\mathsf{X}_M\oplus(Q\rt 0)\oplus(\La\st{1}\rt \La)$ as defined in Theorem \ref{mainTheorem5}.  Then  $\mathsf{Cok}\bigl(\mathsf{Cok}\,\overline{M}\bigr) \subseteq \mathsf{Cok}\ M.$ 
\end{proposition}
\begin{proof}
 For every object $\mathsf{X}$ in  $\mathsf{Cok} \ \overline{M}$  there is a conflation $\mathsf{M}_1\rightarrowtail \mathsf{M}_0\twoheadrightarrow \mathsf{X},$
        where $\mathsf{M}_0, \mathsf{M}_1 \in \mathsf{add}\overline{M}.$ After applying the  cokernel functor, we obtain the following right exact sequence 
        $$\mathsf{Cok} \ \mathsf{M}_1\rt \mathsf{Cok} \ \mathsf{M}_0\rt \mathsf{Cok}\ \mathsf{X}\rt 0.$$
  Since $\mathsf{Cok}\ \overline{M}=M$, we obtain   $\mathsf{Cok}\ \mathsf{M}_1, \mathsf{Cok}\ \mathsf{M}_0 \in \mathsf{add} M $. Hence,  $\mathsf{Cok}\ \mathsf{X} \in \mathsf{Cok}~M$, and consequently $\mathsf{Cok}\bigl(\mathsf{Cok}\,\overline{M}\bigr) \subseteq \mathsf{Cok}\ M.$ 
\end{proof}

We consider in which cases the equality in the above proposition holds.

\begin{proposition}\label{prop. her-eaulity-rigid}
Assume $\La$ is a hereditary algebra. Let  $ \overline{M}=\mathsf{X}_M\oplus(Q\rt 0)\oplus(\La\st{1}\rt \La)$ be  as defined in Theorem \ref{mainTheorem5}.  Then  $\mathsf{Cok}\bigl(\mathsf{Cok}\,\overline{M}\bigr)= \mathsf{Cok}\ M.$ 
\end{proposition}
\begin{proof}
 By Proposition \ref{Propp. inc}, it suffices show   that $ \mathsf{Cok}\ M \subseteq \mathsf{Cok}\bigl(\mathsf{Cok}\,\overline{M}\bigr) $.  
Let $N \in \mathsf{Cok} \  M$. Then, by \cite[Proposition 3.5]{En},  there is a  short  exact sequence in $\mmod \La$
  \[
  \la: 0 \rt \ M_1\rt M_0\rt N\rt 0
  \]
  with $M_1, M_0 \in \mathsf{add} \ M.$ By Proposition \ref{Prop. hereditary.P.right}, the sequence $\la$ is 0-right exat sequence, i.e., $Q_0(\la)=Q_1(\la)=0$. Hence, in view of Construction \ref{Construction5}, there exists  a conflation 
  \[
  \la': \mathsf{K}\rightarrowtail \mathsf{L}\twoheadrightarrow \mathsf{N}
  \]
  with $\mathsf{Cok}\ \la'=\la$. By  the construction, we have $\mathsf{K}=\mathsf{X}_{K} \in \mathsf{add}\ \overline{M}$. Since $\la$ is a $0$-right exact sequence, it follows that   $\mathsf{L}_c=\mathsf{N}_c=0$, and  $$\mathsf{L}_0\oplus\mathsf{L}_a \oplus \mathsf{L}_b \in \mathsf{add} \ \overline{M}.$$ 
  Therefore, $\mathsf{N}$ is the cokernel of a morphism  in $\mathsf{add} \ \overline{M}$. As $\mathsf{Cok} \ \overline{M}$ is closed under cokernls, we infer  that $\mathsf{N} \in \mathsf{Cok} \ \overline{M}$. Consequently, $N=\mathsf{Cok} \ \mathsf{N} \in \mathsf{Cok}\bigl(\mathsf{Cok}\,\overline{M}\bigr)$, as desired.
\end{proof}

By \cite[Theorem 3.2]{En}, every ICE-closed subcategory of $\mmod \La$ with enough $\Ext$-projectives is of the form $\mathsf{Cok}\, M$ for some rigid module $M$. We note that over hereditary algebras, rigid modules coincide with $\tau$-rigid modules. Hence, by Proposition~\ref{prop. her-eaulity-rigid}, it follows that ICE-closed subcategories of $\mmod \La$ with enough $\Ext$-projectives coincide with rICE-closed subcategories of $\mmod \La$ that admit an $r\Ext$-progenerator. 

As a consequence, Theorem~\ref{mainTheorem5} can be regarded as a generalization of the bijection in \cite[Theorem 3.2]{En} to arbitrary Artin algebras, not necessarily hereditary.

In the last section (see Proposition \ref{prop. supp-eaulity-rigid}), we will observe for when $M$ is a support $\tau$-tilting module, the equality in Proposition \ref{Propp. inc} holds. Moreover, we show in Proposition \ref{Prop. cok=fac} that $\mathsf{Cok} \ M=\mathsf{Fac} \ M.$  Therefore, the above theorem can be regarded as an extension of  \cite[Theorem~2.7]{AIR}.

\begin{remark}
  Here, the question arises whether this equality holds in general. Let us explain where the argument breaks down. Let $N \in \mathsf{Cok}\, M$. Then, by \cite[Proposition 3.5]{En}, there exists a short exact sequence in $\mmod \La$
\[
\lambda:\ \ M_1 \rt M_0 \rt N \rt 0
\]
with $M_1, M_0 \in \mathsf{add}\, M$.  By Construction~\ref{Construction5}, there exists a conflation
\[
\lambda':\ \mathsf{K} \rightarrowtail \mathsf{L} \twoheadrightarrow \mathsf{N}
\]
such that $\mathsf{Cok}\, \lambda' = \lambda$. By the construction, we have $\mathsf{K} = \mathsf{X}_{M_1}$, and $
\mathsf{L}_0 \oplus \mathsf{L}_a \oplus \mathsf{L}_b\in \mathsf{add}\, \overline{M}.$
The main difficulty here is that we have no control over the direct summands $\mathsf{L}_c$ and $\mathsf{N}_c$.  In conclusion, if for all right exact sequences $\lambda$ as above the injective objects $(Q_1(\lambda)\rt 0)$ and $(Q_0(\lambda)\rt 0)$ belong to $\mathsf{Cok}\, \overline{M}$, then the desired equality holds.

The bijection established in this section is a reflection of 
Theorem~\ref{Thm. ICE-rigid-Mor}. In Section~\ref{Section 3}, we obtained 
two additional bijections in $\CP(\Lambda)$, namely those in 
Theorems~\ref{Thm. IKE-rigid-Mor} and \ref{thm:main-twin-IEclosed123}. 
It is natural to ask what kinds of bijections in $\mmod \Lambda$ are 
reflected by these results. We plan to investigate these reflections 
in forthcoming work.
\end{remark}

\section{Applications}
  In this section, we prove a bijection between tilting objects in $\CP(\Lambda)$ and support $\tau$-tilting modules in $\mmod \Lambda$. We define r-cotorsion--torsion triples and r-cotorsion--r-torsion triples in $\mmod \Lambda$. Then, as an application of our results from Section~4, we prove our final bijection in this paper between these triples and support $\tau$-tilting modules. We conclude this section with further applications, particularly Theorem~\ref{final}, which highlights the significance of right exact sequences within the framework of support $\tau$-tilting theory.

\subsection{Bijection: tilting objects in $\CP(\La)$ and support $\tau$-tilting modules}
\begin{definition}\label{Def. tilting. CP}
Following  \cite[Definition 4.1 and Theorem 5.3]{S} for the case $n=1$, 
 an additive  subcategory $\mathscr{T}$ of an exact category  $\mathscr{E}$ with enough projectives  is called  {\em tilting} (1-tilting) if it satisfies:
	\begin{enumerate}
		\item $\mathscr{T}$ is closed under summands and self-orthogonal (i.e. $\mathscr{T} \subseteq \mathscr{T}^{\bot}$).
		\item ${\rm pd}_{\mathscr{E}}\mathscr{T}\leq 1$.
		\item For any projective object $P \in \mathscr{E}$, there exists a conflation $P\rightarrowtail C_0\twoheadrightarrow C_1,$ with $C_0, C_1 \in \mathscr{T}.$
	\end{enumerate}
    \end{definition}
If  $\mathscr{T}$ satisfies $(1)$ and $(2)$, then we call it a {\em partial tilting} subcategory. We call an object $T$ in $\mathscr{E}$ a (resp. partial) tilting object if $\mathsf{add}T$ is a  (resp. partial) tilting subcategory.

 We specialize the above definition for the exact category $\CP(\La)$. By noting that $\CP$ has global dimension one, we can infer that an object $\mathsf{M}$ in $\CP$ is a partial tilting object if and only if it is rigid, i.e.,   $\Ext^1_{\CP}(\mathsf{M}, \mathsf{M})=0$. Moreover,  a partial tilting object $\mathsf{M}$ in $\CP(\La)$ is tilting if for any projective object $\mathsf{P}$ in $\CP(\La)$, there exists a conflation
    $$ \mathsf{P}\rightarrowtail \mathsf{M}_0\twoheadrightarrow \mathsf{M}_1$$ such that $\mathsf{M}_0, \mathsf{M}_1\in\mathsf{add}\mathsf{M}$.

We denote by ${\rm tilt}\CP(\La)$  the set of isomorphism classes of basic tilting  objects for $\CP(\La)$.

\begin{lemma}\label{lem. tilt.size}
    Let $\mathsf{T}=(P\st{f}\rt Q)$ be a tilting  object in $\CP(\La)$. Then, $|\mathsf{T}_b|=|\La|$. In particular, $(\La\st{1}\rt \La)$ lies in $\mathsf{add} \ \mathsf{ T}$.
\end{lemma}
\begin{proof}
    Assume $\mathsf{T}=(P\st{f}\rt Q)$ is a tilting object.  Hence,   there exists the following conflation $ \mathsf{X}\rightarrowtail  \mathsf{P}\twoheadrightarrow \mathsf{Q} $, where   $\mathsf{X}=(\La\st{1}\rt \La)$ and $\mathsf{P}, \mathsf{Q} \in \mathsf{add}\mathsf{T}$. Since $\mathsf{X}$ is an injective object in the exact category $\CP$, it is a direct summand of $\mathsf{P}$.  This follows that $\mathsf{X} \in {\rm add}\mathsf{T}$, and therefore   $|\mathsf{T}_b|=|\La|$.
\end{proof}

The proof of the following lemma is inspired by the proof given in \cite[Proposition 4.6]{PZZ}.
\begin{lemma}\label{Grpthen}
   Let $\mathsf{P}=(P\st{d}\rt Q)$ be a tilting object in $\CP(\La)$. Then $|\mathsf{P}|=2| \La |$.
\end{lemma}
\begin{proof}
We begin by noting that there is a triangle equivalence  $\mathbb{D}^{\rm b}(\CP)\simeq \mathbb{K}^{\rm b}({\rm proj}\mbox{-}\CP)$, where  $\mathbb{D}^{\rm b}(\CP)$ is the bounded derived category of the exact category $\CP$, and  ${\rm K}^{\rm b}({\rm proj}\mbox{-}\CP)$ is the homotopy category of bounded complexes of projective objects in $\CP$. This equivalence follows from the fact that $\CP$ is an exact category with finite global dimension. Moreover, by  \cite[Lemma 4.7]{S}, we have a  triangle equivalence
$$ \mathbb{K}^{\rm b}(\mathsf{add}\mathsf{P})\simeq \mathbb{D}^{\rm b} (\CP),$$
\noindent
where $ \mathbb{K}^{\rm b}(\mathsf{add}\mathsf{P})$ is the subcategory of the homotopy category $\mathbb{K}(\CP)$ consisting of all bounded complexes with terms in $\mathsf{add}\mathsf{P}$. Then, by \cite[Lemma 4.1.17]{K22}, we have
$$ {\rm K}_0(\mathsf{add}\mathsf{P})\simeq {\rm K}_0(\mathbb{K}^{\rm b}(\mathsf{add}\mathsf{P}))\simeq {\rm K}_0(\mathbb{D}^{\rm b} (\CP))\simeq {\rm K}_0(\mathbb{K}^{\rm b}({\rm proj}\mbox{-}\CP))\simeq {\rm K}_0(\CP),$$
where the split Grothendieck groups ${\rm K}_0(\mathsf{add}\mathsf{P})$ and ${\rm K}_0(\CP)$ are  free abelian groups with
ranks $|\mathsf{P}|$ and $|\CP|=2 | \La|$, respectively. Thus, the lemma is proved.
\end{proof}

\begin{proposition}\label{Prp. tily-tau-tilt}
 Let $\mathsf{T}$ be an object in $\CP(\La)$, and $M$ a module in $\mmod \La$ with minimal projective presentation $P_1\st{f}\rt P_0\rt M \rt 0$.    The following statements hold.
     \begin{itemize}
         \item [$(i)$]  If $\mathsf{T}$ is a tilting object in $\CP(\La)$, then,  $(\mathsf{Cok}\mathsf{T}, P)$, where $\mathsf{T}_c=(P\rt 0)$, is a support $\tau$-tilting pair in $\mmod \La;$
         \item [$(ii)$] If $(M, Q)$ is a support $\tau$-tilting pair in $\mmod \La, $  then  $(P_1\st{f}\rt P_0)\oplus (Q\rt 0)\oplus (\La\st{1}\rt \La)$ is a tilting object in $\CP.$
     \end{itemize}
\end{proposition}
\begin{proof}
    $(i)$
Assume $\mathsf{T}=(P_1\st{d}\rt P_0)$ is a tilting object in $\CP$, we can observe from Lemma \ref{rigid-tau-rigid} that $\mathsf{Cok}\mathsf{T}$ is a $\tau$-rigid module, where $P$ as stated in the statement. Since $\mathsf{T}$ has no self-extension, we get ${\Ext_{\CP}}((P\rt 0), \mathsf{T})=0$. Then, by Lemma \ref{lem. istau-tilting objects}, we obtain ${\Hom_{\La}}(P, \mathsf{Cok}\mathsf{T})=0$.    Furthermore, Lemma \ref{lem. tilt.size} says that $|\mathsf{T}_b|=|\La|$ and  by Lemma \ref{Grpthen},
$$2|\La|=|\mathsf{T}|=|\mathsf{T}_0|+|\mathsf{T}_a|+|\mathsf{T}_b|+|\mathsf{T}_c|$$
This  implies that $|\mathsf{Cok}\mathsf{T}|+|P|=|\mathsf{T}_0|+|\mathsf{T}_a|+|\mathsf{T}_c|=|\La|$. Therefore, we can conclude that $(\mathsf{Cok}\mathsf{T}, P)$ forms a $\tau$-tilting pair.

$(ii)$  
Let $\mathsf{M}=(P_1\stackrel{f}{\rightarrow} P_0)\oplus (Q\rightarrow 0)\oplus (\Lambda\stackrel{1}{\rightarrow} \Lambda)$. Since $|\mathsf{M}|=2|\Lambda|$, it is sufficient to prove that $\mathsf{M}$ is rigid. According to Subsection \ref{Sub. remainder}, we have,

$$\Ext^1_{\mathscr{P}}(\mathsf{M}, \mathsf{M})\simeq D\underline{{\rm Hom}}_{\mathscr{P}}(\tau^{-1}_{\mathscr{P}}\mathsf{ M}, \mathsf{M})\simeq D\overline{{\rm Hom}}_{\mathscr{P}}(\mathsf{M}, \tau_{\mathscr{P}}\mathsf{M}).$$  Hence, we need to show that $\overline{{\rm Hom}}_{\mathscr{P}}(\mathsf{M}, \tau_{\mathscr{P}}\mathsf{M})=0$. By investigating each type of direct summand of $\mathsf{M}$, we can establish the following:
\begin{itemize}
\item [$(1)$] Let $\mathsf{M}_1:=(P_1\st{f}\rt P_0 )$. It is evident that $\mathsf{Cok}\mathsf{M}_1=M$,
$$\overline{{\rm Hom}}_{\mathscr{P}}(\mathsf{M}, \tau_{\mathscr{P}} \mathsf{M}_1)\simeq \Hom_{\Lambda}(\mathsf{Cok}\mathsf{M}, \mathsf{Cok}\tau_{\mathscr{P}}\mathsf{M}_1)\simeq {\Hom_{\Lambda}}(M, \tau_{\Lambda}M)=0,$$
where the second isomorphism follows from Theorem \ref{Cok-ass}.
     \item [$(2)$] Let $\mathsf{Q}=(Q\rt 0)$. By use of Proposition \ref{Cok-ass},  $\tau_{\mathscr{P}} \mathsf{Q}=(P_1\st{d}\rt P_0)$, with $\mathsf{Cok} d=\nu Q.$ Similar to  the preceding case we have, 
\begin{align*}
     {\Ext^1_{\CP}}( \mathsf{Q}, \mathsf{M}) & \simeq D\overline{{\rm Hom}}_{\mathscr{P}}(\mathsf{M}, \tau_{\mathscr{P}} \mathsf{Q}) \\ 
     & \simeq D\Hom_{\Lambda}(\mathsf{Cok}\mathsf{M}, \mathsf{Cok}\tau_{\mathscr{P}}\mathsf{Q})\\ & = D{\Hom_{\Lambda}}({ M}, \nu Q) \\ & =0,
\end{align*}
          where the last vanishing follows from Lemma \ref{lem. istau-tilting objects}.
\item [$(3)$] Given that $(\Lambda \stackrel{1}{\rightarrow} \Lambda)$ is an injective-projective object in $\CP$, the case is straightforward.
\end{itemize}
\end{proof}

\begin{theorem}\label{Thm. the first-bijection}
Let $\La$ be an Artin algebra.  Then there exist bijections
\[     \begin{tikzcd}[column sep = large]
   {\rm tilt}\CP(\La) \rar["{\mathsf{Cok}}", shift left] & s\tau\mbox{-}{\rm tilt}\La.\lar["{\mathsf{Min}}", shift left]
  \end{tikzcd}     \]

These bijections are  defined as follows: 
 \begin{itemize} 
 \item  $\mathsf{Min}(M, P):=(P_1\st{f}\rt P_0)\oplus (P\rt 0)\oplus ( \La \st{1}\rt \La)$, where $(M, P) \in s\tau\mbox{-}{\rm tilt}\La$ and $P_1\st{f}\rt P_0\rt M\rt 0 $  is a minimal projective presentation of $M$;
 \item $\mathsf{Cok} \mathsf{M}:=(\mathsf{Cok}\mathsf{M}, P')$, where $\mathsf{M}=(P\st{f}\rt Q) \in {\rm tilt}\CP(\La)$ and  $\mathsf{M}_c= (P'\rt 0)$.
 \end{itemize}
 \end{theorem}
 
The cotilting objects in an exact category can be defined dually. In particular, a rigid object $\mathsf{M}$ in the exact category $\CP(\La)$  is said to be cotilting if for every injective object $\mathsf{I}$ in the exact category $\CP(\La)$, there exists a conflation 
$$ \mathsf{M}_1\rightarrowtail \mathsf{M}_0\twoheadrightarrow \mathsf{I},$$
where $\mathsf{M}_1,  \mathsf{M}_0\in\mathsf{add}\mathsf{M}$. In fact, for every object $(P\st{f}\rt Q)$ in $\CP(\La)$ there exists the following conflation 
 $$\xymatrix@1{  (P\st{f}\rt Q)\overset{( \left[\begin{smallmatrix} 1\\ f\end{smallmatrix}\right], ~1)}
			\rightarrowtail
			 (P\rt 0)\oplus (Q\st{1}\rt Q)\overset{([-f~~1]~ 0)}\twoheadrightarrow (Q\rt 0). }$$    
 Therefore, in the definition of cotilting objects in $\CP(\La)$, the injective dimension assumption is automatically satisfied.
Later on,  we will need an analog of Bongartz's lemma in our setting as follows:

\begin{lemma}\label{BongartzLemma}
  Let $\mathsf{T}$ be a partial tilting object in $\CP(\La)$. Then, there exists an object $\mathsf{E}$ in $\CP(\La)$ such that $\mathsf{T}\oplus \mathsf{E}$ is a tilting object in $\CP.$ In particular, $|\mathsf{T}|\leq 2|\La|.$
\end{lemma}
\begin{proof}
Let $\mathsf{Q}=(0 \rt \La) \oplus (\La\st{1}\rt \La).$ $\mathsf{Q}$ is a projective generator for the exact category $\CP.$  Since $\Ext^1_{\CP}(\mathsf{T}, \mathsf{Q})$ is finitely generated as $R$-module, there exist a positive integer $n$ and for each $1\leq i\leq n$, $\la_i: \mathsf{P}\rightarrowtail \mathsf{E}_i\twoheadrightarrow \mathsf{T}$ which generate the $R$-module $\Ext^1_{\CP}(\mathsf{T}, \mathsf{Q})$. We then have the following commutative diagram:

$$\xymatrix{		
		\oplus^n_{i=1} \la_i :& \oplus^{n}_{i=1} \mathsf{Q} \ar[d]^q\ar@{>->}[r] & \oplus^n_{i=1}\mathsf{E}_i
			\ar[d] \ar@{->>}[r] & \oplus^n_{i=1}\mathsf{T} \ar@{=}[d] &\\ \alpha:&	\mathsf{Q}\ar@{>->}[r] & \mathsf{X}
			\ar@{->>}[r] & \oplus^n_{i=1} \mathsf{T}}	$$
where $q$ is the codiagonal morphism.
   A similar argument to the one used to prove  Bongartz's lemma for the module category also works here to show that $\mathsf{T}\oplus \mathsf{X}$ is rigid (see e.g. \cite[\S VI, Lemma 2.4]{ASS}). Then,  the fact that the global dimension of $\CP$ is one shows that $\mathsf{T}\oplus \mathsf{X}$ is the desired tilting object. The last assertion is an immediate consequence of Lemma \ref{Grpthen}. 
\end{proof}

\begin{lemma}\label{Lem. cot=tilt}
An object in $\CP(\La)$ is tilting if and only if so is cotilting. 
\end{lemma}
\begin{proof}
   Assume $\mathsf{T}$ is a cotilting object in $\CP(\La)$. By the dual of Lemma \ref{Grpthen}, we have $\mathsf{T}=2|\La|$. Applying Lemma \ref{BongartzLemma}, we deduce that $\mathsf{T}$ is a tilting object in $\CP(\La)$.  Conversely, suppose $\mathsf{T}$ is a tilting object in $\CP(\La)$. Using the duality $(-)^*$ (see the discussion immediately following Definition~\ref{IKE}), we obtain $(\mathsf{T})^*$ as a cotilting object in $\CP(\La^{\rm op})$, which is also tilting, as shown in the first part of the proof. Therefore, $\mathsf{T}\simeq (\mathsf{T}^*)^*$ is a cotilting object in $\CP(\La)$.
 \end{proof}

Now we are   ready to prove  the result that we promised to show in Section \ref{Section 6}.
\begin{proposition}\label{prop. supp-eaulity-rigid}
Assume $M$ is a support $\tau$-tilting module. Let  $ \overline{M}=\mathsf{X}_M\oplus(Q\rt 0)\oplus(\La\st{1}\rt \La)$ be  as defined in Theorem \ref{mainTheorem5}.  Then  $\mathsf{Cok}\bigl(\mathsf{Cok}\,\overline{M}\bigr)= \mathsf{Cok}\ M.$ 
\end{proposition}
\begin{proof}
By Proposition~\ref{Propp. inc}, it suffices to show that $
\mathsf{Cok}\, M \subseteq 
\mathsf{Cok}\bigl(\mathsf{Cok}\,\overline{M}\bigr).$ Let $N \in \mathsf{Cok}\, M$. Then there exists a right exact sequence in 
$\mmod \Lambda$
\[
\lambda:\quad M_1 \to M_0 \to N \to 0
\]
with $M_1, M_0 \in \mathsf{add}\, M$. By Construction~\ref{Construction5}, there exists a conflation
\[
\lambda':\quad 
\mathsf{K} \rightarrowtail 
\mathsf{L} \twoheadrightarrow 
\mathsf{N}
\]
such that $\mathsf{Cok}\,\lambda' = \lambda$.
By the construction, we have $\mathsf{K} = \mathsf{X}_{M_1}$, and $
\mathsf{L}_0 \oplus \mathsf{L}_a \oplus \mathsf{L}_b 
\in \mathsf{add}\,\overline{M}.$ Theorem~\ref{Thm. the first-bijection} implies that  $\overline{M}$ is a tilting 
object in $\CP$. By Lemma~\ref{Lem. cot=tilt}, it is also cotilting. 
In particular, $(\Lambda \to 0)$ belongs to 
$\mathsf{Cok}\,\overline{M}$. It follows that $\mathsf{L}_c \in \mathsf{Cok}\,\overline{M}.$
Hence $\mathsf{N} \in \mathsf{Cok}\,\overline{M}$, and therefore
\[
N = \mathsf{Cok}\,\mathsf{N} 
\in \mathsf{Cok}\bigl(\mathsf{Cok}\,\overline{M}\bigr),
\]
as desired.
\end{proof}
The following result shows that the bijection in 
Theorem~\ref{mainTheorem5} extends the bijection proved by 
Adachi, Iyama, and Reiten in~\cite{AIR}. 

We remark that this result may can be proved through  the 
module category. However, we choose to prove it using the morphism 
category, as this approach is of independent interest.
\begin{proposition}  \label{Prop. cok=fac}

Let  $M$ be a support $\tau$-tilting module. Then $\mathsf{Cok}\ M=\mathsf{Fac} \ M.$
 \end{proposition}
\begin{proof}
Assume $(M, P)$ is a support $\tau$-tilting pair. It is clear that $\mathsf{Cok}\ M\subseteq \mathsf{Fac}\ M$. So, it is enough to show the other side.  According to Proposition \ref{prop. supp-eaulity-rigid}, $\mathsf{Cok} \ M$ is obtained by taking cokerenel on the ICE-closed subcategory $\CC=\mathsf{Cok} \ \overline{M}$ in $\CP$,  where $\overline{M}$ as defined there.   Since $\overline{M}$ is a tilting object in $\CP$, by Theorem \ref{Thm. the first-bijection}, we have the equality $\mathsf{Cone}(\overline{M},\overline{M})=(\overline{M})^{\perp}$, by Lemma \ref{Lem. tilt.cotilt.prep}. By keeping in mind the chain $\mathsf{Cone} \ \mathsf{N} \subseteq \mathsf{Cok} \ \mathsf{N} \subseteq  \mathsf{Fac} \ \mathsf{N}\subseteq \mathsf{N}^{\perp}$ which holds for every object $\mathsf{N} \in \CP$, we obtain $\mathsf{Cok} \ \overline{M}=\mathsf{Cone} \ (\overline{M}, \overline{M})=\mathsf{Fac} \ \overline{M}.$ Now, we turn to prove the reverse inclusion.  Take $N \in \mathsf{Fac} \ M$. It is not difficult to see that $\mathsf{X}_M$ lies in $\mathsf{Fac}\ \overline{M}$. Hence, $\mathsf{X}_N \in \mathsf{Cok} \ \overline{M}=\mathsf{Cone} \ (\overline{M}, \overline{M})$, consequently, there exists a conflation
$$ \mathsf{M}_1\rightarrowtail \mathsf{M}_0\twoheadrightarrow \mathsf{X}_N, $$
where $\mathsf{M}_1,  \mathsf{M}_0 \in \mathsf{add} \overline{M}$. By taking a cokernel of the above conflation, we get $N\in \mathsf{Cok}M.$
\end{proof}

\subsection{ r-cotorsion pair}
We first recall  the notion of left weak cotorsion pair from \cite{BZ}.
\begin{definition}\label{def. lw-cotrsion-pair}
A pair $(\CC,\CT)$ of subcategories of $\mathrm{mod}\,A$ 
is called a \emph{left weak cotorsion pair} (or \emph{lw-cotorsion pair} 
for short) if the following conditions hold:

\begin{enumerate}
    \item $\Ext^1_{\La}(\CC,\CT) = 0$;
    
    \item For any $M \in \mmod \La$, there exist exact sequences
    \[
    0 \rightarrow Y_M \rightarrow X_M 
    \xrightarrow{\,f_M\,} M \rightarrow 0
    \]
    and
    \[
    M \xrightarrow{\,g_M\,} Y^M \rightarrow X^M \rightarrow 0,
    \]
    where $X_M, X^M \in \CC$, 
    $Y_M, Y^M \in \CT$, 
    $f_M$ is a right $\CC$-approximation of $M$, 
    and $g_M$ is a left $\CT$-approximation of $M$.
\end{enumerate}
\end{definition}
In the sequel, we investigate the cotorsion--torsion triples described in Proposition~\ref{prop. bij-twin-tilting-cotortor}, and determine which triple subcategories of $\mmod \Lambda$ they reflect.

\begin{definition}\label{def. r-cotorsion-torsion}
A pair $(\CC,\CT)$ of subcategories of $\mmod \La$ 
is called a \emph{r-cotorsion pair}  if the following conditions hold:

\begin{enumerate}
    \item ${\Hom_{\La}}(\CC,\tau \CT) = 0$, that is, for every $A \in \CC$ and $B \in \CT$, ${\Hom_{\La}}(A, \tau B)=0$;
    
    \item Let $P$ be the maximal basic projective module such that ${\Hom_{\La}}(P, \CT)=0$. For any $M \in \mmod \La$, there exist $P$-right exact sequences
    \[
      Y_M \rightarrow X_M 
    \rt M \rightarrow 0
    \]
    and
    \[
    M \rt Y^M \rightarrow X^M \rightarrow 0,
    \]
    where $X_M, X^M \in \CC$, 
    $Y_M, Y^M \in \CT$.
    \item Let $Q$ be the maximal basic projective module in $\CC$, that is, the direct sum of one representative from each isomorphism class of indecomposable projective objects in $\CC$.
There is a $Q^*$-right exact sequence in $\mmod \La^{\rm op}$
    \begin{equation}\label{eq. right. exact. r-cot}        
    \La^*\rt \mathsf{Tr} M\oplus P_0\rt \mathsf{Tr} N\oplus P_1\rt 0,
    \end{equation}
    with $M \in \CC, \ N \in \CT$ and $P_0 \in \mathsf{add} P$ and $P_1$ a projective module (where $P$ as defined in $(2)$). 
\end{enumerate}
\end{definition}

In Below, we show that the image of a cotorsion pair as in Proposition \ref{pro. cot-tor.triple} is a r-cotorsion.

\begin{proposition}\label{Prop. r-cotor-in Mor}
Let $\mathsf{T}$ be a tilting object in $\CP(\La)$. Then,   $(\mathsf{Cok}\big( {}^{\perp}(\mathsf{Cok}\mathsf{T})\big),~~\mathsf{Cok}\big(\mathsf{Cok}\mathsf{T}\big))$ is a r-cotorsion pair. 
\end{proposition}
\begin{proof}
  By Theorem~\ref{pro. cot-tor.triple}, we obtain the cotorsion pair 
$({}^{\perp}\mathsf{Cok}\mathsf{T},\, \mathsf{Cok}\mathsf{T})$ in $\CP$. 
Let $P$ be a maximal projective module such that $
{\Hom_{\Lambda}}\bigl(Q, \mathsf{Cok}(\mathsf{Cok}\mathsf{T})\bigr)=0.$
By ARS-duality in $\CP$, as in the proof of Proposition~\ref{Prp. tily-tau-tilt} we did, the vanishing $
\Ext^1_{\CP}\bigl({}^{\perp}\mathsf{Cok}\mathsf{T},\, \mathsf{Cok}\mathsf{T}\bigr)=0$
implies the required Hom-vanishing in Definition~\ref{def. r-cotorsion-torsion}. 
Let $M$ be a module in $\mmod  \Lambda$. 
Since $({}^{\perp}\mathsf{Cok}\mathsf{T},\, \mathsf{Cok}\mathsf{T})$ is a cotorsion pair in $\CP$, 
there exists a conflation
\[
\lambda:\ 
\mathsf{N}\rightarrowtail \mathsf{L}\twoheadrightarrow \mathsf{X}_M,
\]
where $\mathsf{L} \in {}^{\perp}\mathsf{Cok}\mathsf{T}$ and 
$\mathsf{N} \in \mathsf{Cok}\mathsf{T}$.  Let $\mathsf{L}_c = (Q \to 0)$. Since 
${\Ext^1_{\CP}}(\mathsf{L}_c,\, \mathsf{Cok}\mathsf{T}) = 0,$ Lemma~\ref{lem. vanishing-Hom-Ext} yields $
{\Hom_{\Lambda}}\bigl(Q, \mathsf{Cok}(\mathsf{Cok}\mathsf{T})\bigr)=0.$
It follows that $\mathsf{Cok}\,\lambda$ is a $P$-right exact sequence, 
which is one of the required right exact sequences in the definition. 
In a similar manner, we obtain the second right exact sequence required in the definition.

The required right exact sequence (\ref{eq. right. exact. r-cot}) in the definition follows from this fact that for the object $(\La\rt 0)$ in $\CP$ lies in a conflation 
\[
\mathsf{N}\rightarrowtail \mathsf{M}\twoheadrightarrow (\La\rt 0)
\]
with $\mathsf{M} \in  {}^{\perp}(\mathsf{Cok}\mathsf{T})$ and $\mathsf{N} \in \mathsf{Cok}\mathsf{T}$.
By applying the duality functor $(-)^*$ and then taking cokernel, we get the desired right exact sequence.
\end{proof}
 Let $(\CC, \CT)$ be a r-cotorsion pair, and let $P$ be the same projective module in Definition \ref{def. r-cotorsion-torsion}. We introduce two following subcategories in $\CP$:

 $$\mathsf{r}(\CC)=\mathsf{add}\{ \mathsf{X}_M\mid M \in \CC\}\cup \{(P\rt 0), (\La\st{1}\rt \La)\},$$
 and 
 $$\mathsf{r}(\CT)=\mathsf{add}\{ \mathsf{X}_M\mid M \in \CT\}\cup \{(\La\rt 0), (\La\st{1}\rt \La)\}.$$
\begin{proposition}\label{Prop. r-cotrsion-pair-induced}
Let $(\CC,~ \CT)$ be a r-cotorsion pair in $\mmod \La$.  Then, $(\mathsf{r}(\CC),~ \mathsf{r}(\CT))$ is a cotorsion pair in $\CP(\La)$.   
\end{proposition}
\begin{proof}
Let $(\CC,~ \CT)$ be a r-cotorsion pair in $\mmod \La$  together with projective modules $P$ and $Q$ as defined in Definition \ref{def. r-cotorsion-torsion}. First, we  note that both subcategories  $\mathsf{r}(\CC)$ and $ \mathsf{r}(\CT)$ are closed under direct summands. By ARS-duality in conjunction with Lemma \ref{lem. vanishing-Hom-Ext}, the Hom-vanishing   ${\Hom_{\La}}(\CC,\tau \CT) = 0$ implies the required $\Ext$-vanishing in the definition of a cotorsion pair.

Now, let $\mathsf{M}$ be an object in $\CP.$ Since $(\La \st{1}\rt \La)$ lies in both subcategories $\CC$ and $\CT$, we may assume that $\mathsf{M}_b=0$. Set $M=\mathsf{Cok} \mathsf{M}$. First assume that $\mathsf{M}_c=0$. Hence, $\mathsf{M}=\mathsf{X}_M$. By assumption, there is a $P$-right exact sequence 
$$\la: \ M\rt K\rt L\rt 0$$
with $K \in \CT$ and $L \in \CC.$ In view of Construction \ref{Construction5}, we obtain the following conflation in $\CP$
\[
\lambda':\
\mathsf{M}\rightarrowtail \mathsf{K}\twoheadrightarrow \mathsf{L},
\]
where $\mathsf{Cok} \la'=\la$ and $\mathsf{K}_c, \mathsf{L}_c \in \mathsf{add}(P\rt 0)$. Hence, $\mathsf{K} \in r(\CT)$ and $\mathsf{L} \in r(\CC)$. Thus, the conflation $\la'$ gives us one of the conflations in the definition of a cotorsion pair.

Furthermore,  our assumption implies the following $P$-right exact sequence 
  $$\delta: \ F\rt G\rt M\rt 0$$
  with $G \in \CC$ and $F \in \CT.$ In view of Construction \ref{Construction5}, we get the following conflation in $\CP$
\[
\delta':\
\mathsf{F}\rightarrowtail \mathsf{G}\twoheadrightarrow \mathsf{M}',
\]
 with $\mathsf{Cok} \delta'=\delta$ and  $\mathsf{G}_c, \mathsf{M}'_c \in \mathsf{add}(P\rt 0)$. Clearly, $\mathsf{F} \in r(\CT)$. Since $\mathsf{Cok} \mathsf{M}=\mathsf{Cok} \mathsf{M}'=M$, we have $\mathsf{M}'=\mathsf{M}\oplus (P'\rt 0)$, for some $P' \in \mathsf{add} P.$ Consider the following pullback diagram:
 \begin{equation*} 
\xymatrix{\\	 & \mathsf{F} \ar@{=}[d] \ar@{>->}[r] & \mathsf{U} \ar@{>->}[d]\ar@{->>}[r]
 & (P'\rt 0) \ar@{>->}[d]  & \\
	 & \mathsf{F}  \ar@{>->}[r] & \mathsf{G}
	\ar@{->>}[d] \ar@{->>}[r] &   \mathsf{M}' \ar@{->>}[d]  & \\
  &  & \mathsf{M}
	 \ar@{=}[r] & \mathsf{M}    & \\
&  &   &  & }
\end{equation*}
Since $\mathsf{F}$ and $(P'\rt 0)$ are in $r(\CT)$,  the conflation in the first row implies that $\Ext^1_{\CP}(r(\CC), \mathsf{U})=0.$ From the beginning of the proof, there exists the conflation 
\[
\mathsf{U}\rightarrowtail \mathsf{A}\twoheadrightarrow \mathsf{B}
\]
with $\mathsf{A} \in r(\CT)$ and $\mathsf{B} \in r(\CC)$. Hence, this conflation splits. Thus,  $\mathsf{U}$ is a direct summad of $\mathsf{A}$. This implies  that $\mathsf{U} \in r(\CT)$. Therefore, the conflation in the middle column in the above diagram  yields the second required conflation in the definition of a cotorsion pair.

Consider the case where $\mathsf{M}=\mathsf{M}_c$. Thus, $\mathsf{M}=(A\rt 0)$ for some $A \in \mathsf{add} \La$. By the right exact sequence (\ref{eq. right. exact. r-cot}) and Construction \ref{Construction5}, we obtain the following conflation in $\CP(\La^{\rm op})$:
\[
(0\rt A^*)\rightarrowtail \mathsf{M}^*\oplus (0 \rt P_0^*)\twoheadrightarrow \mathsf{N}^*\oplus (0 \rt P_1^*),
\]
where $\mathsf{M} \in r(\CC)$, $\mathsf{N} \in r(\CT)$, $\mathsf{M}^*_c, \mathsf{N}^*_c \in \mathsf{add} Q^*$, and $P_0 \in \mathsf{add} P$. Applying the duality $(-)^*$ yields the desired conflation ending at $\mathsf{M}$. The remaining conflation holds automatically since $\mathsf{M} \in r(\CT)$. This completes the proof.
\end{proof}

We need to determine which pair of subcategories in $\mmod \La$ corresponds to the torsion pair in the cotorsion-torsion triple 
\[
({}^{\perp}(\mathsf{Cok}\,\mathsf{T}),\ \mathsf{Cok}\,\mathsf{T},\ (\mathsf{Cok}\,\mathsf{T})^{\perp_0})
\]
in $\CP(\La)$, where $\mathsf{T}$ is a tilting object.

\begin{definition}\label{Def. r-torsion-pair}
   A pair $(\CC,\CT)$ of subcategories of $\mmod \La$ 
is called a \emph{r-torsion pair}  if 
$$\CC=\mathsf{Cok}(\mathsf{Cok}\mathsf{T}) \ \ \ \ \text{and} \ \ \ \CT=\mathsf{Cok}((\mathsf{Cok}\mathsf{T})^{\perp_0}),$$
where $\mathsf{T}$ is a tilting object in $\CP(\La)$.
\end{definition}
 It should be emphasized that the outer `'$\mathsf{Cok}$'' in the above definition denotes the application of the cokernel functor, while the inner `'$\mathsf{Cok}$'' refers to taking the cokernel of an admissible morphisms in $\mathsf{add}\mathsf{T}$.

The main challenge lies in relating the vanishing of Hom-spaces in $\CP(\La)$ to Hom-vanishing in $\mmod \La$. For this reason, we define an r-torsion pair in this manner. However, in the next result, we provide some necessary conditions for an r-torsion pair.

\begin{proposition}
Let $(\CC, \CT)$ be an r-torsion pair. Then the following conditions hold:
\begin{itemize}
    \item [$(i)$] $\Hom_{\La}(\CC, \CT)=0;$
    \item [$(ii)$] The projective dimension of any module in $\CT$ is at most one;
    \item [$(iii)$] For any module in $\mmod \La$, there exists a $0$-right exact sequence 
    \[
    K \rt M \rt L \rt 0,
    \]
    where $K \in \CC$ and $L \in \CT.$
\end{itemize}
\end{proposition}

\begin{proof}
By definition, $\CC$ and $\CT$ are induced by the torsion pair $(\mathsf{Cok}\,\mathsf{T}, (\mathsf{Cok}\,\mathsf{T})^{\perp_0})$ in $\CP$, where $\mathsf{T}$ is a tilting object in $\CP$. 

$(i)$ This follows directly from the definition of the torsion pair $(\mathsf{Cok}\,\mathsf{T}, (\mathsf{Cok}\,\mathsf{T})^{\perp_0})$. 

$(ii)$ This follows from Proposition \ref{prop. mono-tor-free}. 

$(iii)$ Let $M$ be a module in $\mmod \La$. The torsion pair $(\mathsf{Cok}\,\mathsf{T}, (\mathsf{Cok}\,\mathsf{T})^{\perp_0})$ yields the following conflation in $\CP$:
\[
\mathsf{A} \rightarrowtail \mathsf{X}_M \twoheadrightarrow \mathsf{B},
\]
where $\mathsf{A} \in \mathsf{Cok}\,\mathsf{T}$ and $\mathsf{B} \in (\mathsf{Cok}\,\mathsf{T})^{\perp_0}$. Since $(\La \rt 0) \in \mathsf{Cok}\,\mathsf{T}$, the orthogonality condition implies $\Hom_{\CP}((\La \rt 0), \mathsf{B}) = 0$, which yields $\mathsf{B}_c = 0$. Applying the cokernel functor on the conflation gives the desired right exact sequence.
\end{proof}

An important feature of an r-torsion pair is that it is uniquely determined by the associated torsion pair in $\CP(\La)$. A short proof is provided below.

\begin{lemma}\label{Lem. r-torsion-bij}
    Let $(\CC, \CT)$ be an r-torsion pair  associated to basic tilting object $\mathsf{T}$, see Definition \ref{Def. r-torsion-pair}. If there exists a basic tilting object $\mathsf{T}'$ such that 
    \[
    \CC =\mathsf{Cok}(\mathsf{Cok}\,\mathsf{T}') \quad \text{and} \quad \CT = \mathsf{Cok}((\mathsf{Cok}\,\mathsf{T}')^{\perp_0}),
    \]
    then $\mathsf{T} \simeq \mathsf{T}'$.
\end{lemma}

\begin{proof}
    Since $\mathsf{Cok}(\mathsf{Cok}\,\mathsf{T}) = \mathsf{Cok}(\mathsf{Cok}\,\mathsf{T}')$, the result follows from the bijections in Theorems \ref{mainTheorem5} and \ref{Thm. the first-bijection}.
\end{proof}

 According to the notions of r-cotorsion and r-torsion,  we can define the following two types of triples of subcategories.

\begin{definition}\label{Def. r-cot-r-tor-triple2}
A triple $(\CC, \DD, \CF)$ of subcategories is called:
\begin{itemize} 
    \item an \emph{r-cotorsion-torsion triple} if $(\CC, \DD)$ is an r-cotorsion pair and $(\DD, \CF)$ is a torsion pair;
    \item an \emph{r-cotorsion-r-torsion triple} if $(\CC, \DD)$ is an r-cotorsion pair and $(\DD, \CF)$ is an r-torsion pair.
\end{itemize}
\end{definition}

We also recall from \cite{BZ} a triple of subcategories $(\CC, \DD, \CF)$ is called a \emph{lw-cotorsion-torsion triple} if $(\CC, \DD)$ is a lw-cotorsion pair and $(\DD, \CF)$ is a torsion pair. 

Now we are ready to state the main theorem in this subsection.

\begin{theorem}\label{Thm. r-cot-r-torsion-supp}
There exist bijections between the following sets.
\begin{itemize}
\item [$(1)$] Isomorphism classes of basic support $\tau$-tilting modules;
\item [$(2)$]  r-Cotorsion-torsion triples; 
\item [$(3)$] r-Cotorsion-r-torsion triples;
\item [$(4)$] lw-cotrsion-torsion triples.
\end{itemize}
\end{theorem}
\begin{proof}
From Propositions \ref{Prop. r-cotor-in Mor} and \ref{Prop. r-cotrsion-pair-induced}, we obtain that r-cotorsion pairs in $\mmod \La$ are in bijection with cotorsion pairs in $\CP$. On the other hand, by Lemma \ref{Lem. r-torsion-bij}, we have a bijection between r-torsion pairs in $\mmod \La$ and torsion pairs in $\CP$. Together, these results establish the bijection between $(2)$ and $(3)$. Furthermore, by Theorem \ref{Thm. the first-bijection} and Proposition \ref{prop. bij-twin-tilting-cotortor}, we obtain the first three bijections in the statement. The bijection between $(1)$ and $(4)$ follows from \cite[Corollary 4.7]{BZ}.

\end{proof}

We remark that the subcategories involved in an r-cotorsion pair determine each other, in contrast to an lw-cotorsion pair (see \cite[Example 4.2]{BZ}). In fact, by Proposition \ref{Prop. r-cotrsion-pair-induced}, an r-cotorsion pair is induced by a cotorsion pair in $\CP(\La)$. Since the subcategories involved in a cotorsion pair determine each other, it follows that this property also holds for an r-cotorsion pair. This suggests that the notion of an r-cotorsion pair might be more appropriate than that of an lw-cotorsion pair.

There are other bijections in Proposition \ref{prop. bij-twin-tilting-cotortor} related to torsion-cotorsion triples in $\CP(\Lambda)$. These kinds of triples also reflect a kind of triples in $\mmod \Lambda$ and then we will get more bijections with the class of basic support $\tau$-tilting modules. We skip it and leave it to the reader to formulate it.

\subsection{Further application}

As an application, we present the subsequent characterizations of a $\tau$-rigid pair that serves as a support $\tau$-tilting module.

\begin{theorem}
Let $(M, P)$ be a $\tau$-rigid pair. Then, the following statements are equivalent.
\begin{itemize}
\item [$(1)$] $(M, P)$ is a support $\tau$-tilting pair;
\item [$(2)$] $\mathsf{Cok} M$ satisfies the following properties:
\begin{itemize}
\item [$\bullet$] It is extension closed under right exact sequences, .i.e., for every right exact sequence $L\rt M\rt N\rt 0$ with $L, N \in \mathsf{Cok}M$, we have $M \in \mathsf{Cok}M$;
\item [$\bullet$] For every morphism $f: M\rt N$ in $\mathsf{Cok}M$ with the following factorization

\[\begin{tikzcd}
	M && N \\
	& D
	\arrow["f", from=1-1, to=1-3]
	\arrow["g", two heads, from=1-1, to=2-2]
	\arrow["h", from=2-2, to=1-3]
\end{tikzcd}\]
 we have $D \in \mathsf{Cok}M.$
\end{itemize}

\end{itemize}

\end{theorem}
\begin{proof}
 Assume that $(M, P)$ is a support $\tau$-tilting pair. By Proposition \ref{Prop. cok=fac}, we have  $\mathsf{Fac}M=\mathsf{Cok}M$. Hence, $\mathsf{Cok}M$ is closed under quotients, which trivially implies $(2)$. Conversely,  Assume \( \mathsf{Cok}M \) satisfies the conditions in (2).
Define \[
\CC = \mathsf{add}\{ \mathsf{X}_M \mid M \in \mathsf{Cok}M \} \cup \{ (\La \to 0), (\La \xrightarrow{1} \La) \}.
\]
We can verify that \( \CC \) is an ICE-closed subcategory in \( \CP \). We will show that it is closed under extensions; closed under images and cokernels of admissible morphisms follows similarly. Let \[
 \mathsf{X} \rightarrowtail \mathsf{Y} \twoheadrightarrow \mathsf{Z}
\] be a conflation in \( \CC \). Taking the cokernel of this sequence yields the right exact sequence
\[
 \mathsf{Cok}\mathsf{X} \to \mathsf{Cok}\mathsf{Y} \to \mathsf{Cok}\mathsf{Z} \to 0.
\]
By our assumption, \( \mathsf{Cok}\mathsf{Y} \) belongs to \( \mathsf{Cok}M \) since \( \mathsf{Cok}\mathsf{X}, \mathsf{Cok}\mathsf{Z} \in \mathsf{Cok}M \). Hence, \( \mathsf{Y}_0 \oplus \mathsf{Y}_a \) lies in \( \CC \). Since \( \mathsf{add}(\La \to 0) \oplus (\La \xrightarrow{1} \La) \subseteq \CC \), we do not need to worry about whether the direct summand \( \mathsf{Y}_b \oplus \mathsf{Y}_c \) is in \( \CC \) or not. Now, since $\mathsf{Cok}M=\mathsf{Cok}\CC$, this completes the proof.
\end{proof}

In classical tilting theory, it is known that a partial tilting module \(M\) is considered tilting if there exists a short exact sequence \(0 \to \La \to M_0 \to M_1 \to 0\), where \(M_0, M_1 \in \mathsf{add}M\). In the following, we will explore a similar result in the context of support $\tau$-tilting theory.

\begin{theorem}\label{final}
Let $(M, P)$ be a $\tau$-rigid pair. Then, the following statements hold.
\begin{itemize}
 \item [$(1)$]If $(M, P)$ is  a support $\tau$-tilting pair, then there exists the following exact sequence

 $$0 \rt \Omega^2_{\La}(M_1)\oplus Q\rt \Omega^2_{\La}(M_0)\rt \La\rt M_1\rt M_0\rt 0,$$

 where $M_1, M_0 \in  \mathsf{add}M$, and $Q \in \mathsf{add}P$.
 \item [$(2)$] If there exists a right exact sequence $$\La \to N_1 \to N_0 \to 0,$$
 where $N_1, N_0 \in  \mathsf{add}M$,  then the pair \((M, P)\) can be completed to support \(\tau\)-tilting pair by adding a projective module to the second component. Specifically, there exists a projective module \(P'\) such that \((M, P \oplus P')\) is a support \(\tau\)-tilting pair.

 \end{itemize}
\end{theorem}
\begin{proof}

 $(1)$ Assume that $M$ is support $\tau$-tilting module. We know that $\overline{M}$ is a tilting object in $\CP$, see Theorem \ref{Thm. the first-bijection}. Hence, by definition of a tilting object,  there exists  a conflation
$$(0 \rt \La)\rightarrowtail \mathsf{M}_0\twoheadrightarrow \mathsf{M}_1,$$
where $\mathsf{M}_0, \mathsf{M}_1 \in \mathsf{add}\mathsf{M}$.
Set \(\mathsf{M}'_1 = (\mathsf{M}_1)_0 \oplus (\mathsf{M}_1)_a\). Consider the following pull-back diagram:\[
\xymatrix{
 & (0 \to \La) \ar@{=}[d] \ar@{>->}[r] & \mathsf{M}'_0 \ar[d] \ar@{>>}[r] & \mathsf{M}'_1 \ar[d] &  \\
 & (0 \to \La) \ar@{>->}[r] & \mathsf{M}_0 \ar@{>>}[r] & \mathsf{M}_1  &
}
\]
where \(\mathsf{M}'_1 \to \mathsf{M}_1\) is the inclusion map. Setting \(\mathsf{M}'_0 = (P_1 \xrightarrow{f} P_0)\) and \(\mathsf{M}'_1 = (Q_1 \xrightarrow{g} Q_0)\), the conflation in the first row of the above diagram induces the following commutative diagram in \(\mmod \La\):
\[
\xymatrix{
0 \ar[r] & 0 \ar[d] \ar[r] & P_1 \ar[d]^{f} \ar[r] & Q_1 \ar[d]^{g} \ar[r] & 0 \\
0 \ar[r] & \La \ar[r] & P_0 \ar[r] & Q_0 \ar[r] & 0
}
\]
Applying the snake lemma to the above diagram gives us the desired exact sequence in (1). Note that \(\mathsf{M}'_1\) has no direct summand of form (b) or (c), hence it provides us a minimal projective presentation for \(M_1 = \mathsf{Cok} \mathsf{M}_1'\). Also, since \((\mathsf{M}_0')_c \in \mathsf{add}(P \to 0)\), we deduce that the projective module \(Q\) appearing in the desired exact sequence lies in \(\mathsf{add} P\).

$(2)$ Assume that there is a right   exact sequence
\begin{equation}\label{seq5sect1}
\La\st{f}\rt N_1\st{g}\rt N_0\rt 0
\end{equation}
with $N_1, N_0 \in \mathsf{add}M.$ Let $P_1\st{h}\rt P_0\st{d}\rt N_0\rt 0$ be a minimal projective presentation. Consider the pullback of the sequence (\ref{seq5sect1}) along $d$, and complete
it to the following commutative diagram with exact rows and columns:

{\small $$\xymatrix{&&0\ar[d]&0\ar[d]&\\	 &  & \Omega_{\La}(N_0) \ar[d]^j
\ar@{=}[r] &  \Omega_{\La}(N_0)\ar[d]^i & \\
	0 \ar[r] & \La\ar@{=}[d] \ar[r]^{\left[\begin{smallmatrix} 1\\  0\end{smallmatrix}\right]} & \La\oplus P_0
	\ar[d]^{[f~~b]} \ar[r]^{\left[\begin{smallmatrix} 0 & 1\end{smallmatrix}\right]} & P_0 \ar[d]^d \ar[r]\ar@{.>}[dl]_b & 0\\
 & \La \ar[d] \ar[r]^f & N_1
	\ar[d] \ar[r]^g & N_0 \ar[d] \ar[r] & 0\\
& 0  & 0  & 0 & }
$$}
\noindent
Here $i$ denotes the inclusion in the factorization $h:P_1\st{l}\rt \Omega_{\La}(N)\st{i}\rt P_0.$ By using the above commutative diagram, we can build the following conflation
\begin{equation}\label{eq.lem55}
(0 \rt \La) \st{(0~~\left[\begin{smallmatrix} 1\\  0\end{smallmatrix}\right] )}\rightarrowtail (P_1\st{jl}\rt \La \oplus P_0 ) \st{(1~~[0~~1])}\twoheadrightarrow (P_1\st{h}\rt P_0)
\end{equation}
\noindent
Set $\mathsf{X}=(P_1\st{jl}\rt \La \oplus P_0 )$ and $\mathsf{N}=(P_1\st{h}\rt P_0)$.  The same proof as in Proposition \ref{Prp. tily-tau-tilt} shows that $\Ext^1_{\CP}(\mathsf{N}, \mathsf{N})=0$. Consequently, $\mathsf{N}'=\mathsf{N}\oplus (\La\st{1}\rt \La)$ is a rigid object.  Using the long exact sequence associated with conflation \ref{eq.lem55}, we can infer that $\Ext^1_{\La}(\mathsf{X}_c, \mathsf{N}')=0$.  Since $\mathsf{Cok}\mathsf{X}=N_1$, we have $\mathsf{N}_0\oplus \mathsf{N}_a \in \mathsf{add}\mathsf{N}$. Let $\mathsf{N}''=\mathsf{N}\oplus\mathsf{X}_c\oplus(\La\st{1}\rt \La)$. Finally, $\mathsf{N}''$ is a rigid object in $\CP$,  and based on conflation \ref{eq.lem55}, for every projective object in $\CP$ we have the desired conflation in $\CP$ for the definition of a tilting object. Hence, $\mathsf{N}''$  is a tilting object in $\CP$. According to Theorem \ref{Thm. the first-bijection}, $(\mathsf{Cok}\mathsf{N}'', Q)=(M, Q)$, where $\mathsf{X}_c=(Q\rt  0)$,  is a support $\tau$-tilting module. Since $Q$ is a maximal projective module such that ${\Hom}_{\La}(Q, M)=0$, it follows that  $P\in \mathsf{add}Q$. This completes the proof.
\end{proof}

\section*{Acknowledgments}
The research of the second author was in part supported by a grant from IPM (No. 1403160416). The work of the second author is
based upon research funded by Iran National Science Foundation (INSF) under project No. 4001480. The third author is partially supported by NSFC Grant No. 12271249.


\begin{thebibliography}{9999}
\bibitem[AIR]{AIR} {\sc T. Adachi, O. Iyama and I. Reiten,} {\sl {$\tau$}-tilting theory,} Compos. Math., \textbf{150}(3) (2014), 415-452.



\bibitem[ASS]{ASS} {\sc I. Assem, D. Simson and A. Skowro\'nski,} {\sl  Elements of the representation theory of associative algebras,} Vol. 1, volume 65 of London Mathematical Society Student Texts. Cambridge
University Press, Cambridge, 2006. Techniques of representation theory.

	
\bibitem[A1]{Aus} {\sc M. Auslander,} {\sl Representation dimension of Artin algebras,} Queen Mary College Notes, 1971.





\bibitem[ARS]{ARS} {\sc M. Auslander, I. Reiten and S. O. Smal\o,} {\sl Representation theory of Artin algebras,} Cambridge Studies in Advanced Mathematics, 36. Cambridge University Press, Cambridge, 1995. xiv+423 pp. ISBN: 0-521-41134-3.


Addendum: J. Algebra \textbf{71} (1981), 592--594.

\bibitem[Ba]{Ba} {\sc R. Bautista,} {\sl The category of morphisms between projective modules,} Comm. Algebra, {\bf 32}(11) (2004), 4303-4331.


\bibitem[BB]{BB} {\sc S. Brenner and M. C. R. Butler,} {\sl Generalizations of the Bernstein-Gelfand-Ponomarev reflection functors.}
In Representation theory, II (Proc. Second Internat. Conf., Carleton Univ., Ottawa, Ont., 1979), volume 832 of Lecture Notes in Math., pages 103–169. Springer, Berlin-New York, 1980.

\bibitem[BZ]{BZ} {\sc A. B. Buan and Y. Zhou,} {\sl  Weak cotorsion, $\tau$-tilting and two-term categories,} J. Pure Appl. Algebra {\bf 228} (2024), no. 1, Paper No. 107445, 18 pp.


\bibitem[Bu]{Bu} {\sc T. B\"uhler,} {\sl  Exact categories,} Expo. Math., {\bf 28}(1) (2010), 1-69.

\bibitem[DF]{DF} {\sc H. Derksen and J. Fei,} {\sl General presentations of algebras,} Adv. Math., {\bf 278} (2015), 210–237.


 \bibitem[E]{En} {\sc H. Enomoto,} {\sl  Rigid modules and ICE-closed subcategories in quiver representations,}  J. Algebra, {\bf 594} (2022), 364-388.

\bibitem[ES1]{ES} {\sc  H. Enomoto and  A. Sakai,} {\sl  ICE-closed subcategories and wide $\tau$-tilting modules,} Math. Z., {\bf 300}(1) (2022), 541–577.

\bibitem[ES2]{ES2}{\sc H. Enomoto and A. Sakai,} {\sl Image-extension-closed subcategories of module categories of hereditary algebras,}  J. Pure Appl. Algebra {\bf 227} (2023), no. 9, Paper No. 107372, 17 pp.

\bibitem[FZ]{FZ} {\sc S. Fomin and A. Zelevinsky,} {\sl Cluster algebras. I. Foundations,} J. Amer. Math. Soc., {\bf15}(2) (2002), 497–
529.

\bibitem[GNP]{GNP1} {\sc M. {Gorsky}, H. {Nakaoka} and Y. {Palu},} {\sl   Hereditary extriangulated categories: Silting objects, mutation, negative extensions,} \newblock {\em  arXiv:2303.07134.}




\bibitem [HE]{HE} {\sc R. Hafezi and  H. Eshraghi,} {\sl From morphism categories to functor categories,} Bull. Malays. Math. Sci. Soc. {\bf 48} (2025), no. 3, Paper No. 88, 39pp.

\bibitem[HNW]{HNW} {\sc R. Hafezi, A. R. Isfahani and Jiaqun Wei,} {\sl $\tau$-tilting theory via the morphism category of projective modules I: ICE-closed subcategories,} available on arXiv:2410.17965.


 \bibitem [HW]{HW} {\sc R. Hafezi and Jiaqun Wei,} {\sl Auslander-Reiten theory on the morphism category of projective modules}, \newblock {\em arXiv:2307.10715,} accepted for publication in Communications in Mathematics and Statistics.

\bibitem[HR]{HR} {\sc D. Happel and C. M. Ringel,} {\sl Tilted algebras,} Trans. Amer. Math. Soc., {\bf274}(2) (1982), 399-443.

\bibitem[HU]{HU} {\sc D. Happel and L. Unger,} {\sl Almost complete tilting modules,} Proc. Amer. Math. Soc., 107 (1989), 603-610.



\bibitem[IT]{IT} {\sc C. Ingalls and H. Thomas,} {\sl Noncrossing partitions and representations of quivers,} Compos.
Math., {\bf145}(6) (2009), 1533–1562.



 \bibitem[INP]{INP} {\sc O. Iyama, H. Nakaoka and Y. Palu,} {\sl Auslander-Reiten theory in extriangulated categories,} Trans. Amer. Math. Soc. Ser. B 11 (2024), 248–305.






\bibitem[K]{K22} {\sc H. Krause,} {\sl  Homological theory of representations,} Cambridge Studies in Advanced Mathematics, vol. 195, Cambridge University Press, Cambridge, 2022.


\bibitem[PZZ]{PZZ} {\sc J. Pan, Y. Zhang and  B. Zhu,} {\sl Support $\tau$-tilting subcategories in exact categories,} J. Algebra, {\bf 636} (2023), 455–482.


\bibitem[RS]{RS} {\sc C. Riedtmann and A. Schofield,} {\sl On a simplicial complex associated with tilting modules,}
Comment. Math. Helv., {\bf 66} (1991), 70-78.


\bibitem[MS]{MS} {\sc F. Marks and J. {\v{S}}t'{\i}v\'{\i}{\v{c}}ek,} {\sl Universal localizations via silting,}  Proc. Roy. Soc. Edinburgh Sect. A, vol. 149, no. 2, pp. 511--532, 2019.


\bibitem[S]{S} {\sc J. Sauter,} {\sl  Tilting theory in exact categories,} \newblock {\em arXiv: 2208.06381.}


\bibitem[Un]{Un} {\sc L. Unger,} {\sl Schur modules over wild, finite-dimensional path algebras with three simple modules,}J. Pure Appl. Algebra, {\bf 64} (1990), 205-222.


\end{thebibliography}
\end{document}